\documentclass[12pt]{amsart}
\usepackage{amssymb,amsfonts,amsmath,amsthm}
\usepackage[a4paper,margin=1in]{geometry}
\allowdisplaybreaks

\usepackage{lmodern}
\usepackage{microtype}
\usepackage[english]{babel}
\usepackage{mathrsfs}
\usepackage{mathtools}
\usepackage{enumerate}
\usepackage{hyperref}
\usepackage{xcolor}
\theoremstyle{plain}
\newtheorem{theorem}{Theorem}[section]
\newtheorem{lemma}[theorem]{Lemma}

\newtheorem{conjecture}{Conjecture}
\newtheorem{corollary}[theorem]{Corollary}

\numberwithin{equation}{section} 
\theoremstyle{definition}

\DeclareMathOperator{\sinc}{sinc}

\newcommand{\Z}{\mathbb{Z}}

\newcommand{\Q}{\mathbb{Q}}
\newcommand{\R}{\mathbb{R}}

\newcommand{\C}{\mathbb{C}}

\newcommand{\sgn}{\mathrm{sgn}}
\newcommand{\re}{\operatorname{Re}}
\newcommand{\im}{\operatorname{Im}}

\renewcommand{\phi}{\varphi}

\newcommand{\Cone}{\mathscr{C}}

\begin{document}
\title[The H\"{o}rmander--Bernhardsson extremal function ]
{The H\"{o}rmander--Bernhardsson extremal function}

\author[A. Bondarenko]{Andriy Bondarenko}
\address{Department of Mathematical Sciences \\ Norwegian University of Science and Technology \\ NO-7491 Trondheim \\ Norway}
\email{andriybond@gmail.com}

\author[J. Ortega-Cerd\`{a}]{Joaquim Ortega-Cerd\`{a}}
\address{Department de Matem\`{a}tiques i Inform\`{a}tica, Universitat de
Barcelona, Gran Via 585, 08007 Barcelona, Spain and
CRM, Centre de Recerca Matemàtica, Campus de Bellaterra Edifici C, 08193
Bellaterra, Barcelona, Spain}
\email{jortega@ub.edu}

\author[D. Radchenko]{Danylo Radchenko}
\address{Laboratoire Paul Painlev\'e \\ Universit\'e de Lille \\ F-59655 Villeneuve d'Ascq \\ France}
\email{danradchenko@gmail.com}

\author[K. Seip]{Kristian Seip}
\address{Department of Mathematical Sciences, Norwegian University of Science and Technology (NTNU), 7491 Trondheim, Norway}
\email{kristian.seip@ntnu.no}

\thanks{Bondarenko and Seip were supported in part by Grant 334466 of the
Research Council of Norway. Ortega-Cerd\`a has been supported by grants
PID2021-123405NB-I00 and CEX2020-001084-M by the Agencia Estatal de
Investigacion and by the Departament de Recerca i Universitats, grant 2021 SGR
00087 and ICREA Academia. Radchenko acknowledges funding by the European Union 
(ERC, FourIntExP,
101078782).}

\begin{abstract}
We characterize the function $\varphi$ of minimal $L^1$ norm among all functions $f$ of exponential type at most $\pi$ for which $f(0)=1$. This function, studied by
H\"{o}rmander and Bernhardsson in 1993, has only real zeros $\pm \tau_n$, $n=1,2, \ldots$. Starting from the fact that $n+\frac12-\tau_n$ is an $\ell^2$ sequence, established in an earlier paper of ours, we identify $\varphi$ in the following way. We factor $\varphi(z)$ as $\Phi(z)\Phi(-z)$, where $\Phi(z)= \prod_{n=1}^\infty(1+(-1)^n\frac{z}{\tau_n})$ and show that $\Phi$ satisfies a certain  second order linear differential equation along with a functional equation either of which characterizes $\Phi$. We use these facts to establish an odd power series expansion of $n+\frac12-\tau_n$ in terms of $(n+\frac12)^{-1}$ and a power series expansion of the Fourier transform of $\varphi$, as suggested by the numerical work of H\"{o}rmander and Bernhardsson. The dual characterization of $\Phi$ arises from a commutation relation that holds more generally for a two-parameter family of differential operators, a fact that is used to perform high precision numerical computations. 
\end{abstract}
\maketitle
\tableofcontents

\section{Introduction}
The H\"{o}rmander--Bernhardsson extremal function is the unique solution $\varphi$ to the extremal problem
\begin{equation}\label{eq:extremalproblem}
	\frac{1}{\Cone} = \inf_{f \in PW^1} \left\{\|f\|_{L^1(\mathbb R)} \,: \, f(0) = 1\right\},
\end{equation}
where
$\Cone$  is the smallest positive constant $C$ such that the inequality 
\begin{equation}\label{eq:pointeval}
	|f(0)| \leq C \|f\|_{L^1(\mathbb R)}
\end{equation}
holds for every $f$ in the Paley--Wiener space $PW^1$, i.e., the subspace of $L^1(\mathbb{R})$ consisting of entire functions of exponential type at most $\pi$. The problem of computing
$\Cone$ and identifying $\varphi$ has attracted the attention of workers in operator theory \cite{BDK89}, partial differential equations~\cite{HB93}, approximation theory \cite{AKP96, Gorbachev05}, orthogonal polynomials \cite{LL15, Lubinsky17}, and number theory \cite{CMS19}, and it appears to be of some basic interest in Fourier analysis.

In  \cite{HB93}, H\"{o}rmander and Bernhardsson obtained the impressive numerical approximation
\begin{equation}\label{eq:hb}
	0.5409288219 \leq \Cone \leq 0.5409288220,
\end{equation}
but added that, unfortunately, they had not been able to identify the extremal function using
the numerical data, as originally hoped for. Building on our preliminary work performed in \cite{BORS}, we will in the present paper remedy this situation and give a comprehensive description of $\varphi$ which, as we will see,  stands out as a rather remarkable special function.

We start from the fact that
\[ \varphi(z) = \prod_{n=1}^\infty \left(1-\frac{z^2}{\tau_n^2}\right),\]
with $\tau\coloneqq (\tau_n)_{n\geq1}$ a strictly increasing sequence of positive numbers
(see \cite{HB93, LL15, BCOS}) and $n+\frac12-\tau_n$ belonging to $\ell^2$ (see
\cite{BORS}). The main characters in what follows will be the constant~$\Cone$, which can
be expressed as
\begin{equation} \label{eq:expressC}  \Cone = \frac{1}{2}\prod_{n=1}^\infty \frac{\tau_n^2}{n(n+1)} \quad \text{or} \quad \frac{1}{\Cone}=2+4\sum_{n=1}^\infty (-1)^n \Big(n+\frac12-\tau_n\Big), \end{equation}
the constant
\[ L_{\tau}(1)\coloneqq \sum_{n=1}^\infty \frac{(-1)^n}{\tau_n} , \]
and the entire function
\[  \Phi(z)\coloneqq \prod_{n=1}^\infty\left(1+(-1)^n\frac{z}{\tau_n}\right). \]
Note that $\varphi$ factors as $\varphi(z)=\Phi(z)\Phi(-z)$ and also that $\Phi'(0)=L_\tau(1)$.

Our main result can now be stated as follows.

\begin{theorem}\label{thm:main}
The function $\Phi$ is a solution to the differential equation
\begin{equation} \label{eq:diffbasic} z^2f''(z) + \Big(2z-\frac{1}{2\Cone}\Big)f'(z) + \Big(\frac{\pi^2}{4}z^2+\frac{L_{\tau}(1)}{2\Cone}\Big)f(z) = 0 \end{equation}
and the functional equation
\begin{equation} \label{eq:functional} F(z) e^{\frac{1}{4\Cone z}}=
\frac{e^{-i\frac{\pi}{4}+i\frac{\pi}{2}z} F\big(\frac{1}{2\pi i \Cone
z}\big)+e^{i\frac{\pi}{4}-i\frac{\pi}{2}z} F\big(-\frac{1}{2\pi i \Cone
z}\big)}{2\sqrt{\pi \Cone}z} ,\qquad z\in \C\smallsetminus \{0\} . \end{equation}
Conversely, an entire function solving either \eqref{eq:diffbasic} or \eqref{eq:functional} is a complex scalar times $\Phi$.
\end{theorem}

This dual characterization of the function $\Phi$ is rather unexpected. Compatibility 
of~\eqref{eq:diffbasic} and~\eqref{eq:functional} owes itself to the fact that the differential operator from~\eqref{eq:diffbasic} commutes with an action of the Klein four-group on holomorphic functions on $\C^{*}:=\C\smallsetminus\{0\}$. More precisely, the differential operator in question commutes with the two involutions sending $f(z)$ to $e^{-\frac{1}{4\Cone z}\pm i\frac{\pi}{2}z}f(\pm\frac{1}{2\pi i\Cone z})$. These involutions define a $\Z/2\Z\times \Z/2\Z$-action on holomorphic functions on $\C^{*}$, and, because of the commutation relation, the equations~\eqref{eq:diffbasic} and~\eqref{eq:functional} end up sharing a nonzero solution. 

It is worth pointing out a curious similarity with another instance of time-frequency localization, namely the classical Landau--Pollak--Slepian theory of time-and-band limiting  and the associated prolate spheroidal wave functions (see \cite{SP61, LP61, LP62, Slepian83}). Of central importance in that subject is the surprising existence of a differential operator---the prolate spheroidal operator---that commutes with the integral operator of time-and-band limiting. We refer to Gr\"{u}nbaum's commentary \cite{Grunbaum} on Connes and Moscovici's remarkable evolution of this theory \cite{Connes} for an overview of the substantial research efforts made to understand and extend the Landau--Pollak--Slepian theory.  As in that story, the commutation relation at hand, though easily verified, appears at first glance rather coincidental. While it is crucial to our analysis, we lack for the moment a conceptual explanation for why the commutation occurs in our case.

The differential and functional equations complement each other by giving different kinds of information about 
$\Phi$ and its zeros. The differential equation~\eqref{eq:diffbasic} shows for example that 
all coefficients of the Taylor expansion of $\varphi$ are polynomials with rational 
coefficients in $\pi^2$, $\Cone$, and $L_{\tau}(1)$ (see formula \eqref{eq:phialg} below). 
On the other hand, we use~\eqref{eq:functional} to obtain a series 
expansion for the zeros $\tau_n$ and to establish strong regularity of the 
Fourier transform $\widehat{\phi}$. The former result reads as follows.
\begin{theorem}\label{thm:zeros}
The numbers $\tau_n$ can be expanded as
	\begin{equation} \label{eq:taunformula}
	\tau_n = n + 1/2 - \rho\left(\frac{1}{n+1/2}\right),
	\end{equation}
where $\rho(z) = \sum_{m\ge 1} a_m z^m$ is an odd power series whose domain of convergence is $|z|\le 2$, with $a_m\ge 0$ for all $m\ge 1$ and $\rho(2)=\frac12$.
 \end{theorem}
The curious point that $\rho(2)=\frac12$ suggests the existence of an ``invisible'' zero at $0$,  which is consistent with the appearance of the point $0$ in the orthogonality relations used in \cite[Lem. 3.3 and Sec. 8.2]{BCOS}. The coefficients $a_n$ are also rational functions of $\pi^2$, $\Cone$, and $L_{\tau}(1)$ (see formulas \eqref{eq:a1q} and \eqref{eq:allaq} below), as can be seen to follow from the differential equation.

To state our result concerning the regularity of $\widehat{\phi}$, we begin by defining the Fourier transform of an $L^1$ function $f$ as
\[ \widehat{f}(\xi)\coloneqq \int_{-\infty}^\infty f(x) e^{-i\xi x} dx. \]
With this convention, functions in $PW^1$ have Fourier transforms supported on $(-\pi,\pi)$.
To place our result in context, we give a brief description of H\"{o}rmander and Bernhardsson's numerical work. The core of their approach
was to use the multivariate version of Newton's method to solve the original extremal problem in the 4-dimensional subspace of $PW^1$ consisting of~$f$ such that on $[-1,1]$, $\widehat{f}(\pi \xi)$ is a polynomial in $1-\xi^2$ of degree at most $4$. The Fourier transform of the corresponding extremal function $f_{\text{HB}}$ was thus of the form
\[ \widehat{f}_{\text{HB}}(\pi \xi) = c_1(1-\xi^2)+c_2(1-\xi^2)^2+c_3(1-\xi^2)^3+c_4(1-\xi^2)^4 , \quad \xi \in [-1,1]. \]
H\"{o}rmander and Bernhardsson went even further and computed similar approximations for polynomials in $1-\xi^2$ of degree at most $6$, but noted that the first 14 decimals of the approximation of $\Cone$ would not change.  In spite of the success of their approach, they could not conclude that $f_{\text{HB}}$ is a good approximation to $\varphi$. Indeed, H\"{o}rmander and Bernhardsson remarked that they had no proof that $\xi\mapsto \widehat{\varphi}(\pi\xi)$
is infinitely differentiable in $[- 1, 1 ]$, although the fast convergence of the
numerical approximations even suggested analyticity  \cite[p. 92]{HB93}. The following result confirms what H\"{o}rmander and Bernhardsson were hinting at.
\begin{theorem} \label{thm:phifourier}
We have
$\widehat{\phi}(\pi \xi) = h (\xi)$ for $\xi$ in $[-1,1]$, where $h$ is the entire function of order $\frac12$ defined by the power series
	\begin{equation} \label{eq:phifourier}
	h(z)\coloneqq \pi\sum_{n=1}^\infty  \frac{(\Phi^2)^{(n-1)}(0)}{(n-1)! n!} \Big( \frac{1-z}{2\Cone}\Big)^n .
	\end{equation}
\end{theorem}
Since $h$ is an even function, it has an alternate power series representation of the form
\[ h(z)=\sum_{n=1}^\infty c_n (1-z^2)^n .\]
In retrospect, we see that H\"{o}rmander and Bernhardsson's work yields approximations to the first four coefficients  in this expansion. We may now verify that the decay of the coefficients~$c_n$ is indeed very fast, as it  can be quantified by using that  $w\mapsto \sum_{n=1}^\infty c_n w^n$ is an entire function of order $\frac14$ and finite type.

\subsection*{Outline of the paper} In the next section, we deduce from the familiar description of $\varphi$ in terms of $L^1$ orthogonality several alternate characterizations, one of which reveals that 
$A(z)\coloneqq e^{\frac{1}{4\Cone z}} \Phi(z)$ satisfies the quadratic differential equation
\[ A'(z)A(-z)+A'(-z)A(z)=-\frac{1}{2\Cone z^2}. \]
This relation opens up the path to the linear differential equation \eqref{eq:diffbasic} which is presented in \S~\ref{sec:ode}. The proof of the part of Theorem~\ref{thm:main} pertaining to that equation is then completed in this section.

The salient features of the differential equation \eqref{eq:diffbasic} are  best understood by taking the point of view that $-\frac{L_\tau(1)}{2\Cone}$ is an eigenvalue and $\Phi$ an associated eigenfunction of the second order linear operator $\mathcal{L}_{a,b}$ which acts on $f$ by the rule
\[ \mathcal{L}_{a,b}(f)(z) \coloneqq z^2f''(z) + (2z-a)f'(z) + b^2z^2f(z)\, ,\] 
with $a=\frac{1}{2\Cone}$ and $b=\frac{\pi}{2}$. In \S~\ref{sec:family}, we take a closer look at this family of operators when $ab\neq 0$ and establish the important fact that $\mathcal{L}_{a,b}$ commutes with the two operators $U_{\pm}$ defined as follows:
		\[U_{\pm}f(z) = z^{-1}e^{\mp i bz -\frac{a}{2z}}f\Big(\pm
\frac{ia}{2bz}\Big).\]
We identify the spectrum of $\mathcal{L}_{a,b}$, in part as preparation for high precision numerical computation of $\Cone$ and $L_\tau(1)$.

In \S~\ref{sec:functionaleq}, we turn to the second part of the proof of Theorem~\ref{thm:main}. The fact that $\Phi$ solves \eqref{eq:functional} is an easy consequence of the work done in \S~\ref{sec:family}. We then apply a fixed point argument to deal with the remaining difficulty of showing that all solutions are multiples of $\Phi$. The next two sections \S~\ref{sec:taun} and \S~\ref{sec:phihat} use the functional equation to establish Theorem~\ref{thm:zeros} and Theorem~\ref{thm:phifourier}, respectively.

In \S~\ref{sec:summation}, we return to one of the consequences of the description of $\varphi$ in terms of $L^1$ orthogonality, namely the summation formula
\[
-\frac{f'(0)}{2\Cone} = \sum_{n= 1}^\infty (-1)^n(f(\tau_n)-f(-\tau_n)),
\]
valid for all $f$ in $PW^1$. We show that we get a similar formula associated with the zeros of each of the eigenfunctions of $\mathcal{L}_{a,\frac{\pi}2}$, and so we obtain an abundance of such formulas. 

We describe in \S~\ref{sec:numerics} an algorithm for computing $\Cone$ and $L_{\tau}(1)$ to any desired precision. To 100-digit precision, we find that 
\begin{equation} \label{eq:precise100}
\begin{aligned}
	\Cone =\;
	0.&5409288219018305893928820589996903868550226554937\\
	&596979480680071379607164927092133820213498385139732\ldots\\
	L_{\tau}(1) =\;
	{-}0.&4519521648844099974932868451365782916108606506737\\
	&789537658475272354499485397456507823153397452618492\ldots,
	\end{aligned}
\end{equation}
and we see that $\Cone$ is in agreement with the lower bound in \eqref{eq:hb}. 

In the final section, we deduce the two formulas in \eqref{eq:expressC}, before presenting some curious problems that may await future investigation. The first of these arises from a consideration of the two Dirichlet series
\[L_{+}(s) \coloneqq \sum_{n=1}^{\infty}\frac{1}{\tau_n^s}\qquad\mbox{and}\qquad L_{-}(s) \coloneqq  \sum_{n=1}^{\infty}\frac{(-1)^n}{\tau_n^s},\]
the latter of which extends to an entire function and the former to a meromorphic function in $\mathbb C$. We have verified numerically to very high precision that $L_{+}(-2k)=(2\pi i \Cone)^{-2k} L_+(2k)$ for all integers $k$, and the challenge is to prove this relation.  The second problem is to verify an integrality phenomenon, namely that the Taylor coefficients of $\varphi$ appear to lie in $\Z[\pi, L_{\tau}(1), \Cone]$. These problems suggest that there may be additional  structures related to the H\"ormander--Bernhardsson function yet to be disclosed.

\section{Characterizations of the sequence \texorpdfstring{$\tau$}{tau}}
\label{sec:quaddiffeq}

We will say that a strictly increasing sequence $t=(t_n)_{n=1}^\infty$ is an admissible sequence of positive numbers (or simply an admissible sequence) if all the $t_n$ are positive and $(n+\frac12-t_n)$ is in $\ell^2$.
To any admissible sequence $t=(t_n)_{n=1}^\infty$, we associate the two functions
\begin{equation} \label{eq:theta} \psi_t(z)\coloneqq \prod_{n=1}^\infty \left(1-\frac{z^2}{t_n^2}\right) \quad \text{and} \quad \Theta_t(z)\coloneqq -\frac{1}{4\Cone z}+\frac{z}{2}\sum_{n=1}^\infty (-1)^n\Big(\frac{1}{t_n-z}+\frac{1}{t_n+z}\Big). \end{equation}
We will now deduce the following characterization of the sequence $\tau$ and hence of the extremal function $\varphi=\psi_\tau$.
\begin{theorem}\label{thm:charphi}  Let $t$ be an admissible sequence. Then $t=\tau$ if and only if
\begin{equation} \label{eq:charphi}
	 \psi_t(z) \Theta_t(z)=-\frac{1}{4\Cone z}.
\end{equation}
\end{theorem}
The main conclusion of this section, of crucial importance in the sequel, is that Theorem~\ref{thm:charphi} can be rephrased as saying that the function
	\[ A_t(z) \coloneqq \exp\Big(\frac{1}{4\Cone z}\Big)\prod_{n=1}^\infty \Big(1+(-1)^n\frac{z}{t_n}\Big) \]
	satisfies a first order quadratic differential equation when $t=\tau$. Indeed, observing that
	\[ \Theta_t(z)= \frac{z}{2}\left(\frac{A_t'(z)}{A_t(z)}+\frac{A_t'(-z)}{A_t(-z)}\right) \]
and using the fact that $\psi_t(z)=A_t(z)A_t(-z)$, we see that Theorem~\ref{thm:charphi} can be recast as the following assertion.
\begin{corollary}\label{thm:diff} Let $t$ be an admissible sequence. Then $t=\tau$ if and only if
	\begin{equation} \label{eq:Aeq}
	A_t'(z)A_t(-z)+A_t'(-z)A_t(z)=-\frac{1}{2\Cone z^2}.
	\end{equation}
\end{corollary}

The proof of Theorem~\ref{thm:charphi} will rely on two equivalent characterizations of the sequence~$\tau$, the first of which is tied more directly to our original extremal problem \eqref{eq:extremalproblem}.
\begin{lemma} \label{lem:summform}
Let $t$ be an admissible sequence. Then $t=\tau$ if and only if
\begin{equation}  \sum_{n=1}^\infty  (-1)^{n}\sin (t_n \xi) = - \frac{\xi}{4\Cone}  \label{eq:dercond}  \end{equation}
for $-\pi < \xi < \pi$ in the sense of tempered distributions.
\end{lemma}

\begin{proof}
We start from the fact that
\begin{equation} \label{eq:basicid} \frac{f(0)}{\Cone}=\int_{-t_1}^{t_1} f(x) dx + \sum_{n=1}^\infty (-1)^n \int_{t_n}^{t_{n+1}} (f(x)+f(-x)) dx \end{equation}
holds for every $f$ in $PW^1$ if and only if $t=\tau$ (see \cite[Thm. 3.6]{BCOS}).
Using that
\[ f(0)=  \int_{-\infty}^\infty f(x) \frac{\sin \pi x}{\pi x} dx\]
and the fact $\mathcal{S}\cap PW^1$ is dense in $PW^1$,
we see by Plancherel's identity that \eqref{eq:basicid} holds for all $f$ in $PW^1$ if and only if
\begin{equation} \label{eq:Fourierbasic}  \sum_{n=1}^\infty  (-1)^n \frac{\sin t_n \xi}{\xi} = - \frac{1}{4\Cone}, \quad -\pi < \xi < \pi ,\end{equation}
holds in the sense of tempered distributions. It is clear that \eqref{eq:Fourierbasic} implies \eqref{eq:dercond}. On the other hand, since the distributional Fourier transform of $\sgn \cos(\pi x)$
vanishes on $(-\pi,\pi)$ (see \cite[p. 188]{HB93}),  we have
\begin{equation}\label{eq:F0}  \sum_{n=0}^\infty  (-1)^{n}\frac{\sin \big(n+\frac12\big) \xi}{\xi} =0 , \quad -\pi < \xi < \pi. \end{equation} Hence \eqref{eq:dercond} implies that
\[  \sum_{n=1}^\infty  (-1)^{n}\big(\sin t_n \xi-\sin \big(n+\frac12\big) \xi \big) =  \sin \frac{\xi}{2} - \frac{\xi}{4\Cone}, \quad -\pi < \xi < \pi. \]
Setting $d_n\coloneqq n+\frac12-t_n$, we may write the left-hand side of this identity as
\[  \sum_{n=1}^\infty  (-1)^{n}\big(\big(\cos d_n \xi - 1\big) \sin \big(n+\frac12\big) \xi -\sin d_n \xi \cos\big(n+\frac12\big)  \xi \big) \]
from which we see that
\[  \frac{1}{\xi} \sum_{n=1}^\infty  (-1)^{n}\big(\sin t_n \xi-\sin \big(n+\frac12\big) \xi \big) \]
is an $L^2$ function by the assumption that $(d_n)$ is in $\ell^2$. This means that \eqref{eq:dercond} yields
\begin{equation} \label{eq:sumC}  \sum_{n=1}^\infty  (-1)^{n}\Big(\frac{\sin t_n \xi}{\xi}-\frac{\sin \big(n+\frac12\big) \xi}{\xi} \Big) =  \frac{\sin \frac{\xi}{2}}{\xi} - \frac{1}{4\Cone}, \quad -\pi < \xi < \pi. \end{equation}
Using again \eqref{eq:F0}, we finally arrive at \eqref{eq:Fourierbasic}.
\end{proof}

Our second characterization of $\tau$  shows that the identity of the preceding lemma can be recast as a reproducing formula for even functions in $PW^1$.
\begin{lemma} \label{lem:frep} Let $t$ be an admissible sequence. Then $t=\tau$ if and only if
\begin{equation} \label{eq:frepr} f(z) \Theta_t(z) =- \frac{f(0)}{4\Cone z} + \sum_{n=1}^{\infty} (-1)^n \frac{t_n z}{(t_n^2-z^2)} f(t_n)\end{equation}
for every even function $f$ in $PW^1$.
\end{lemma}
\begin{proof}
We set
\[ L_t (2m-1)\coloneqq \sum_{n=1}^\infty \frac{(-1)^n}{t_n^{2m-1}} \]
for $m\ge 1$ and $L_t(-1)\coloneqq -\frac{1}{4\Cone}$ so that
\[ \Theta_t(z)=\sum_{m=0}^\infty L_t(2m-1) z^{2m-1} \]
when $0<|z|<t_1$.

We begin by assuming that \eqref{eq:dercond} holds. We observe that this identity can be written as
\[ L_t(1)-\frac{L_t(-1)}{2}\xi^2 =\sum_{n=1}^\infty \frac{(-1)^n}{t_n} \cos (t_n \xi), \quad -\pi < \xi < \pi ,\]
which now holds pointwise, by comparison with Fourier series in terms of $\cos(n+\frac12)\xi $  and $\sin(n+\frac12)\xi$.
Anti-differentiating iteratively, we get the formula
\[ L_t(2m-1)-\frac{L_t(2m-3)}{2}\xi^2 + \cdots + (-1)^m \frac{L_t(-1)}{(2m)!} \xi^{2m} = \sum_{n=1}^\infty \frac{(-1)^n}{t_n^{2m-1}} \cos (t_n \xi),  \]
which is valid for $-\pi<\xi <\pi$.
Integrating both sides of this identity against an even function $\widehat{f}$ in $\mathcal{S}$ supported on $[-\pi,\pi]$, we see that
\begin{equation} \label{eq:mth} L_t(2m-1)f(0)+\frac{L_t(2m-3)}{2}f''(0) + \cdots + \frac{L_t(-1)}{(2m)!} f^{(2m)}(0) = \sum_{n=1}^\infty \frac{(-1)^n}{t_n^{2m-1}} f(t_n). \end{equation}
The precaution that $\widehat{f}$ in $\mathcal{S}$ is only needed when $m=1$, but in any case,  this formula holds for all even functions $f$ in $PW^1$ by the Plancherel--P\'{o}lya inequality and the fact that $\mathcal{S}\cap PW^1$ is dense in $PW^1$.
We observe that the left-hand side of \eqref{eq:mth} is the $(2m-1)$th coefficient of the Laurent expansion of $\Theta_t(z)f(z)$ about $0$. Since
\[ \sum_{m=1}^\infty \sum_{n=1}^\infty \frac{(-1)^n}{t_n^{2m-1}} f(t_n) z^{2m-1} = \sum_{n=1}^{\infty} (-1)^n \frac{t_n z}{(t_n^2-z^2)} f(t_n), \]
we arrive at \eqref{eq:frepr}.

We finally observe that \eqref{eq:frepr} implies \eqref{eq:dercond}, simply because all steps in the above deduction can be reversed.
\end{proof}

Before turning to the proof of Theorem~\ref{thm:charphi}, we record an approximation result based on standard arguments from the theory of Paley--Wiener spaces. We associate with every admissible sequence $t$ a function
\[ \Psi_t(z)\coloneqq \prod_{n=1}^\infty \Big(1+(-1)^n \frac{z}{t_n}\Big). \] We will let $ t^\ast$ denote the symmetric sequence $(\pm t_n)$ (the zero set of $\psi_t$) and $t^{\circ}=((-1)^{n+1}t_n)$ be the zero set of $\Psi_t$.

\begin{lemma} \label{lem:adm}
If $t=(t_n)_{n=1}^\infty$ is an admissible sequence, then
\begin{equation} \label{eq:psit} \Psi_t(z) \asymp \frac{e^{\frac{\pi}{2}|y|}}{(1+|z|)} \frac{\operatorname{dist}(z,t^\circ)}{(1+\operatorname{dist}(z,t^\circ))} , \quad z=x+iy\in \mathbb{C} . \end{equation}
Moreover, if $f$ is in $PW^1$ and vanishes on $t^\ast$, then $f$ is a constant multiple of $\psi_t$.
\end{lemma}

\begin{proof}
We set $d_n\coloneqq n+\frac12-t_n$ and express $\Psi_t$ in the following way:
	\begin{equation} \label{eq:psirep}
	\Psi_t(z) = C \frac{\sin \frac{\pi}{2}(z+\frac12)}{\frac{\pi}{2}(z+\frac12)} \prod_{n=1}^\infty \left(1-\frac{d_n}{n+\frac{1}{2}+(-1)^n z} \right),
	\end{equation}
where
\begin{equation} \label{eq:Cdef}  C\coloneqq \prod_{n=1}^\infty  \frac{\big(n+\frac12\big)}{t_n}. \end{equation}
Since $(d_n)$ is in $\ell^2$, we see by an application of the Cauchy--Schwarz inequality that the infinite product in \eqref{eq:psirep} is $\asymp 1$ when the distance from $z$ to the two sequences $t^\circ$ and $((-1)^{n+1}(n+\frac12))$ exceeds, say, $\frac14$. Hence the two-sided bound \eqref{eq:psit} holds for such $z$. We then extend this bound to all $z$ by an application of the maximum modulus principle in the union of the discs of radius $\frac14$ around the points $(-1)^{n+1}t_n$ and $(-1)^{n+1}(n+\frac12)$.

The second part of the lemma follows from what was just proved. Indeed, an $f$
in $PW^1$ vanishing on $t^\ast$ must be of the form $\omega\psi_t$ for some entire function $\omega$ of exponential type
$0$. Then either $\omega$ is
a polynomial or it has infinitely many zeros. In the latter case, consider the
function $g(z) \coloneqq \omega(z)/((z-a)(z-b))$ where $a,b$ are two zeros of $\omega$.
Since $f$ is bounded on the real line,  $g$ is also bounded there in view of \eqref{eq:psit}. But being of exponential type $0$, $g$ must then be a constant.
We are therefore left with the possibility that $\omega$ is a polynomial. Using again \eqref{eq:psit} and the assumption that $f$ is in
$L^1(\mathbb R)$, we conclude that $\omega$ must be a constant.
\end{proof}

\begin{proof}[Proof of Theorem~\ref{thm:charphi}]
In view of Lemma~\ref{lem:frep}, we need to show that \eqref{eq:charphi} holds if and only if \eqref{eq:frepr} holds for all $f$ in $PW^1$. The implication from \eqref{eq:frepr} to  \eqref{eq:charphi} is trivial since it is a matter of setting $f=\psi_t$. To prove the reverse implication, we begin by assuming that \eqref{eq:charphi} holds. We let $f$ be an arbitrary even function in $PW^1$.
We may assume without loss of generality that $f(x)\ll (1+x^2)^{-1}$ when $|x|\to \infty$, since functions of this kind form a dense subset of $PW^1$, and set
\[ F(z)\coloneqq f(z)\Theta_t(z)+\frac{f(0)}{4\Cone z}-\sum_{n=1}^\infty (-1)^n \frac{t_n z}{(t_n^2-z^2)} f(t_n), \]
which is seen to be an entire function. Hence $F\psi_t$ is an entire function that vanishes at $t_n$ for all $n\ge1$. Using \eqref{eq:charphi}, we find that
\begin{equation} \label{eq:Fth} F(z) \psi_t (z) = -\frac{f(z)}{4\Cone z}+\psi_t(z)\left(\frac{f(0)}{4\Cone z}-\sum_{n=1}^\infty (-1)^n \frac{t_n z}{(t_n^2-z^2)} f(t_n)\right) \end{equation}
from which we see that $F\psi_t$ is in $PW^1$. By Lemma~\ref{lem:adm}, $F$ must be a constant function. But \eqref{eq:Fth} implies that
\[ F(x)\psi_t(x) \ll \frac{1}{1+|x|^3} + \frac{|\psi_t(x)|}{1+|x|}, \]
and this can only hold if $F(x)\equiv 0$, since $|\psi_t(x)| \asymp \operatorname{dist}(x,t)/(1+|x|)^2$ in view of Lemma~\ref{lem:adm}.
  \end{proof}

\section{The second order linear differential equation of Theorem~\ref{thm:main}}
\label{sec:ode}
We will first show that Corollary~\ref{thm:diff} implies \eqref{eq:diffbasic}. To this end, we introduce the notation
\[ \mathcal{L}f(z)\coloneqq z^2f''(z) + \Big(2z-\frac{1}{2\Cone}\Big)f'(z) + \Big(\frac{\pi^2}{4}z^2+\frac{L_{\tau}(1)}{2\Cone}\Big)f(z). \]
We let $PW_{\frac{\pi}{2}}^\infty$ denote the Bernstein space consisting of entire functions of exponential type at most $\frac{\pi}{2}$ that are bounded on the real line. We will require the following lemma.
\begin{lemma}\label{lem:Bernstein}
The function $\mathcal{L}\Phi$ is in $PW_{\frac{\pi}{2}}^\infty$.
\end{lemma}
\begin{proof}It suffices to check that the function
\begin{equation} \label{eq:simpler} G(z)\coloneqq z^2\Phi''(z) + 2z\Phi'(z) + \frac{\pi^2}{4}z^2\Phi(z) \end{equation}
belongs to $PW_{\frac{\pi}{2}}^\infty$. We set $\delta_k\coloneqq \Phi(2k+\frac12)$ and start from the expansion
\[ \Phi(z)=\sum_{k\in \mathbb{Z}} \delta_k  \sinc \big(\frac{\pi}{2}(z-2k+\frac12)\big).	 \]
Then
\[ \Phi'(z)=\frac{\pi}{2} \sum_{k\in \mathbb{Z}}\delta_k \frac{\cos\big(\frac{\pi}{2}\big(z-2k+\frac12\big)\big)}{\frac{\pi}{2}\big(z-2k+\frac12\big)}-\sum_{k\in \mathbb{Z}}\delta_k \frac{\sin\big(\frac{\pi}{2}\big(z-2k+\frac12\big)\big)}{\frac{\pi}{2}\big(z-2k+\frac12\big)^2}  \]
and
\[ \Phi''(z)=-\frac{\pi^2}{4} \Phi(z)-\pi \sum_{k\in \mathbb{Z}}\delta_k \frac{\cos\big(\frac{\pi}{2}\big(z-2k+\frac12\big)\big)}{\frac{\pi}{2}\big(z-2k+\frac12\big)^2}+2\sum_{k\in \mathbb{Z}}\delta_k \frac{\sin\big(\frac{\pi}{2}\big(z-2k+\frac12\big)\big)}{\frac{\pi}{2}\big(z-2k+\frac12\big)^3} . \]
Plugging these expressions into \eqref{eq:simpler} and simplifying, we find that
\begin{align*} G(z)= & 2 \sum_{k\in \mathbb{Z}}\delta_k z\big(\frac12-2k\big)\frac{\cos\big(\frac{\pi}{2}\big(z-2k+\frac12\big)\big)}{\big(z-2k+\frac12\big)^2} \\
& -\frac{4}{\pi}
\sum_{k\in \mathbb{Z}}\delta_k z\big(\frac12-2k\big)\frac{\sin\big(\frac{\pi}{2}\big(z-2k+\frac12\big)\big)}{\frac{\pi}{2}\big(z-2k+\frac12\big)^3}.  \end{align*}
By Lemma~\ref{lem:adm}, $\delta_k=\Phi(2k+\frac12)=O(\frac{1}{k^2+1})$ when $|k|\to \infty$. Hence each of the two latter sums can be estimated trivially, and we thus find that they are bounded for real $x$ at a positive distance from the points $\frac12-2k$. Then $G(x)$ itself is bounded on the real line, since we may deal with points near $\frac12-2k$ by using Taylor expansions of respectively
$\cos\big(\frac{\pi}{2}\big(z-2k+\frac12\big)\big)$ and $\sin\big(\frac{\pi}{2}\big(z-2k+\frac12\big)\big)$ around $\frac12-2k$.
 \end{proof}

\begin{proof}[Proof of~\eqref{eq:diffbasic}]
Our goal is now to show that $\mathcal{L} \Phi$ vanishes on the set $\{0\}\cup\{(-1)^{n+1}\tau_n\}_{n=1}^\infty$. We get immediately $\mathcal{L}\Phi(0)=0$ because $\Phi'(0)=L_\tau(1)$. As to the assertion that $\mathcal{L}\Phi((-1)^{n+1}\tau_n)=0$, we need to check that
\begin{equation} \label{eq:crucial} \tau_n^2 \Phi''\big((-1)^{n+1}\tau_n\big)=\Big(2(-1)^{n}\tau_n+\frac{1}{2\Cone}\Big) \Phi'\big((-1)^{n+1}\tau_n\big) \end{equation}
holds for every $n\ge 1$. To this end, we start from \eqref{eq:Aeq} which we write as
\begin{equation} \label{eq:quadratic} z^2 \big(\Phi'(z)\Phi(-z)+\Phi'(-z)\Phi(z)\big)-\frac{1}{2\Cone} \Phi(z)\Phi(-z)=-\frac{1}{2\Cone}. \end{equation}
Differentiating this equation and setting $z=(-1)^{n+1}\tau_n$, we find that
\[ \Phi\big((-1)^n \tau_n\big)\Big(\tau_n^2 \Phi''\big((-1)^{n+1} \tau_n \big)+\big(2(-1)^{n+1}\tau_n-\frac{1}{2\Cone}\big)\Phi'\big((-1)^{n+1} \tau_n \big)\Big)=0,  \]
which yields \eqref{eq:crucial} since $\Phi((-1)^n\tau_n)\neq 0$ for every $n\ge 1$.

By the definition of $\Phi$ and what was just shown, $\frac{\mathcal{L}\Phi}{z\Phi}$ is an entire function. Invoking
Lemma~\ref{lem:Bernstein} along with Lemma~\ref{lem:adm}, we may therefore employ Liouville's theorem to infer that
\[ \mathcal{L} \Phi (z)= C z \Phi(z) \]
for some constant $C$. But the derivative of  $\mathcal{L} \Phi (z)$ at $0$ is
\[ 2\Phi'(0)-\frac{1}{2\Cone}\Phi''(0)+\frac{(\Phi'(0))^2}{2\Cone}, \]
and we see that this equals $0$ by differentiating  \eqref{eq:quadratic} twice and evaluating at $0$. Hence $C=0$ since $\Phi(0)=1$ , and we conclude that
	\[ \mathcal{L} \Phi (z)=0. \qedhere \]
\end{proof}
Since $\phi(z)=\Phi(z)\Phi(-z)$, the differential equation $ \mathcal{L} \Phi (z)=0$ implies the following.
\begin{corollary}\label{cor:thirdorder}
The H\"ormander--Bernhardsson function $\phi(z)$ satisfies
	\[\phi'''(z) + \frac6z\phi''(z) +
	\Big(\pi^2+\frac{6\Cone+2L_{\tau}(1)}{\Cone
		z^2}-\frac{1}{4\Cone^2z^4}\Big)\phi'(z) +
	\Big(\frac{2\pi^2}{z}+\frac{2L_{\tau}(1)}{\Cone z^3}\Big)\phi(z)= 0\,. \]
\end{corollary}
\begin{proof}
	From $\mathcal{L}\Phi(z)=0$ we see that the function $g(z)=e^{\frac{1}{4\Cone z}}\Phi(z)$ satisfies
	\begin{equation}\label{eq:gdiffeq1}
		g''(z)+\frac{2}{z}g'(z) + \Big(\frac{\pi^2}{4}+\frac{L_{\tau}(1)}{2\Cone z^2}-\frac{1}{16\Cone ^2z^4}\Big)g(z) 
		= 0.
	\end{equation}
	It is easy to see that $g(-z)$ satisfies the same differential equation. Then $\phi(z)=g(z)g(-z)$ satisfies a third order differential equation that is the symmetric square of the differential operator in~\eqref{eq:gdiffeq1}, and computing it directly gives the claim.
\end{proof}

Note that the differential equations yield a recursion relation for the Taylor coefficients of $\Phi$ and $\phi$ which therefore can be written as polynomials  in $\pi$, $\Cone$, and $L_{\tau}(1)$ with rational coefficients. The first few terms in these Taylor expansions are
	\[	
	\Phi(z) = 1+L_{\tau}(1)z + \Big(\frac{L_{\tau}(1)^2}{2}+2L_{\tau}(1)\Cone\Big)z^2
	+\Big(\frac{L_{\tau}(1)^3}{6}+\frac{8}{3}L_{\tau}(1)^2\Cone+8L_{\tau}(1)\Cone^2+\frac{\pi^2\Cone}{6}\Big)z^3 + \cdots\]
and
	\begin{equation} \label{eq:phialg} \phi(z) = 1 + 4\Cone L_{\tau}(1)z^2 + \big(96L_{\tau}(1)\Cone^3+(24L_{\tau}(1)^2+2\pi^2)\Cone^2\big)z^4+\cdots\,.\end{equation}

We finally note that the recursion relation satisfied by the Taylor coefficients of any solution $f$ of \eqref{eq:diffbasic} (see \eqref{eq:recursion} below) shows that $f$ is uniquely determined by its value at $0$ and hence must be a complex scalar times $\Phi$. 

\section{A family of differential operators and a commutation relation}
\label{sec:family}
For $a,b$ in $\C^{*}$, consider the family of linear differential equations
$\mathcal{L}_{a,b}f=\lambda f$, where
\begin{equation} \label{eq:diffspecial}
	\mathcal{L}_{a,b}(f)(z) := z^2f''(z) + (2z-a)f'(z) + b^2z^2f(z)\,.
\end{equation}
This differential equation has irregular singularities at $0$ and at $\infty$ and no
other singularities in $\C^{*}$. If $R_k$ denotes the rescaling $(R_kf)(z)\coloneqq f(kz)$, then we have $\mathcal{L}_{a,b}R_k=R_k\mathcal{L}_{ka,k^{-1}b}$, so up to conjugation $\mathcal{L}_{a,b}$ only depends on the product $ab$. We will also need the computation
	\begin{equation} \label{eq:diffspecialalt}
	e^{\frac{a}{2z}}\mathcal{L}_{a,b}(e^{-\frac{a}{2z}}g) = z^2g''(z)+2zg'(z)+\Big(b^2z^2-\frac{a^2}{4z^2}\Big)g(z)
	\end{equation}
analogous to~\eqref{eq:gdiffeq1}.

We will call $\lambda$ in $\C$ an eigenvalue of $\mathcal{L}_{a,b}$ if
$\mathcal{L}_{a,b}f=\lambda f$ admits a nonzero solution $f(z)$ holomorphic in
a neighborhood of $0$, and we will call a corresponding solution with $f(0)=1$ its (normalized) $\lambda$-eigenfunction. We denote by $\sigma(\mathcal{L}_{a,b})$ the set of eigenvalues of~$\mathcal{L}_{a,b}$. The coefficients of the expansion $f(z)=\sum_{n\ge0}\alpha_nz^n$ of any eigenfunction $f$ 
are uniquely determined by the recursion
\begin{equation} \label{eq:recursion}
	a(n+1)\alpha_{n+1} = (n(n+1)-\lambda)\alpha_n+b^2\alpha_{n-2}\,,
	\qquad n\ge0\,,
\end{equation}
together with the initial condition $\alpha_0=1$ (we also set $\alpha_{-1}=\alpha_{-2}=0$).
Therefore, $\lambda$ is in $\sigma(\mathcal{L}_{a,b})$ if and only if the sequence
$\alpha_n$ grows at most exponentially. Moreover, a solution that is holomorphic around zero automatically extends to an entire function since $\mathcal{L}_{a,b}$ has no singularities in
$\C^{*}$.

\subsection{Commuting operators and a functional equation}
A key property of the differential operators $\mathcal{L}_{a,b}$ is the following commutation relation.
\begin{lemma} \label{lem:commrel}
	The operators $U_{\pm}=U_{\pm,a,b}$, defined by
		\[U_{\pm}f(z) \coloneqq  z^{-1}e^{\mp i bz -\frac{a}{2z}}f\Big(\pm
\frac{ia}{2bz}\Big)\]
	commute with $\mathcal{L}_{a,b}$.
\end{lemma}
\begin{proof}
	In terms of $g(z)\coloneqq e^{\frac{a}{2z}}f(z)$, the claim is equivalent to the assertion that
	\[g\mapsto  z^2g''+2zg'+\Big(b^2z^2-\frac{a^2}{4z^2}\Big)g\]
	commutes with $g\mapsto z^{-1}g(\pm\frac{ia}{2bz})$. It is easy to check that $g\mapsto z^2g''+2zg'$ commutes with any rescaling $R_k$ as well as with $g\mapsto z^{-1}g(\frac1z)$, so it commutes with $g\mapsto z^{-1}g(\pm\frac{ia}{2bz})$. On the other hand, the mapping $g\mapsto (b^2z^2-\frac{a^2}{4z^2})g$ commutes with $g\mapsto z^{-1}g(\frac{k}z)$ if and only if $a^2+4k^2b^2=0$. Combining these facts, we obtain the desired conclusion.
\end{proof}

\begin{lemma} \label{lem:feqgeneral}
	Any nonzero eigenfunction $f$ of $\mathcal{L}_{a,b}$ satisfies the functional equation
	\begin{equation} \label{eq:feqgeneral}
	ze^{\frac{a}{2z}}f(z) = \kappa_{+}e^{-i bz}f\Big(\frac{ia}{2bz}\Big) + \kappa_{-}e^{i bz}f\Big(-\frac{ia}{2bz}\Big)
	\end{equation}
	for some nonzero constants $\kappa_{\pm}$ with $\kappa_{+}^2-\kappa_{-}^2=\frac{ia}{2b}$.
\end{lemma}
\begin{proof}
	The three functions $f$, $U_{+}f$, and $U_{-}f$ are nonzero solutions of a second order differential equation and thus linearly dependent. It is easy to see that $U_{+}f$ and $U_{-}f$ cannot be proportional, so the relation can be written in the form~\eqref{eq:feqgeneral} for some constants $\kappa_{\pm}$. Next, if we had $\kappa_{-}=0$, then we would have
		\[e^{i bz}f(z) = \kappa_{+}z^{-1}e^{-\frac{a}{2z}}f\Big(\frac{ia}{2bz}\Big)\,,\]
	and this would imply that $z\mapsto e^{ibz}f(z)$ is nonzero, entire, and goes to $0$ as $|z|\to\infty$, in conflict with Liouville's theorem. So $\kappa_{-}\ne0$, and similarly $\kappa_{+}\ne0$ by the same argument.

	Writing $g(z)\coloneqq e^{\frac{a}{2z}}f(z)$, we see that the equation takes the form
	\[zg(z) = \kappa_{+}g\Big(\frac{ia}{2bz}\Big) + \kappa_{-}g\Big(-\frac{ia}{2bz}\Big).\]
	Considering Laurent expansions of both sides, we get for all $n$ in $\Z$
	\[ a_{n-1} = a_{-n}\Big(\frac{ia}{2b}\Big)^{-n}(\kappa_{+}+\kappa_{-}(-1)^n),\]
	where $g(z)=\sum_{n\in \Z}a_nz^n$. Comparing this with the equation for $n\mapsto-(n-1)$ and using the fact that not all coefficients $a_n$ vanish, we get
	\begin{equation} \label{eq:involution} 
	\kappa_{+}^2-\kappa_{-}^2=\frac{ia}{2b}\,.\qedhere
	\end{equation}
\end{proof}

\subsection{Spectrum of \texorpdfstring{$\mathcal{L}_{a,b}$}{L_{a,b}}}
\label{sec:spectrum}
Up to this point we have only seen one eigenfunction of $\mathcal{L}_{a,b}$, namely $\Phi(z)$  for $a=\frac{1}{2\Cone}$, $b=\frac{\pi}2$. However, for any choice of $a,b\in\C^{*}$ the spectrum of $\mathcal{L}_{a,b}$ is infinite, and so there are plenty of eigenfunctions to which the above discussion applies. In fact, standard arguments may be employed to show that for suitable choices of the parameters $a,b$, these eigenfunctions constitute a Riesz basis for the Paley--Wiener space $PW^2_b$, i.e., the subspace of $L^2(\mathbb R)$ consisting of entire functions of exponential type at most $b$.

We consider $\mathcal{L}_{a,b}$ as an unbounded operator acting on $PW^2_b$. Since $\mathcal{L}_{a,b}$ is conjugate to $\mathcal{L}_{ab,1}$, we may assume that $b=1$. Let us first look at the differential operator
\[ L_0 g(\xi)\coloneqq \frac{d}{d\xi}(\xi^2-1)g'(\xi) \]
on $[-1,1]$ which has eigenvalues $ n(n+1)$ with the Legendre polynomials $P_n$ the associated eigenfunctions . We are interested in the perturbation
	\[ L_a g(\xi)\coloneqq L_0g(\xi)-ia\xi g(\xi),\]
which satisfies\footnote{We remark that $L_{a}\widehat{f}=\lambda \widehat{f}$ is a special case of the confluent Heun equation~\cite[Sec. 31.12]{NIST:DLMF}} $L_a\widehat{f} = \widehat{\mathcal{L}_{a,1}f}$.
We note that by Bonnet's recursion formula,
\[ L_a P_n =n(n+1) P_n-ia\big(\frac{n+1}{2n+1} P_{n+1}+\frac{n}{2n+1}P_{n-1}\big). \]
This means that if  $g=\sum_{n=0}^\infty i^n\xi_n P_n$ and the $\xi_n$ decay sufficiently fast, then
\[ L_a g= \sum_{n=0}^\infty n(n+1) i^n\xi_n P_n - a \sum_{n=0}^\infty i^n \big(\frac{n}{2n-1}\xi_{n-1}-\frac{n+1}{2n+3}\xi_{n+1}\big) P_n , \]
where we use the convention that $\xi_{-1}=0$. We now wish to find out for which $\lambda$ the system of equations
	\begin{equation} \label{eq:tridiag}
	\lambda \xi_n = n(n+1) \xi_n - a \big(\frac{n}{2n-1}\xi_{n-1}-\frac{n+1}{2n+3}\xi_{n+1}\big) 
	\end{equation}
has a solution that decays for $n\to\infty$. This means that the eigenvalues of our differential equation are the eigenvalues of the infinite tridiagonal matrix $T$ with entries that on the $m$th row are
	\begin{equation} \label{eq:entries} 
		-a \frac{m}{2m-1}, \quad m(m+1), \quad  a \frac{m+1}{2m+3}.
	\end{equation}

\begin{theorem}\label{thm:spectrum}
	For any $a,b$ in $\C^{*}$ the spectrum of $\mathcal{L}_{a,b}$ is an infinite discrete closed set. Moreover, all sufficiently large eigenvalues (depending on~$a$ and~$b$) can be labeled as $\{\lambda_n\}_{n\ge N}$ with $\lambda_n=n(n+1)+O(1/n)$ as $n\to\infty$.
\end{theorem}
\begin{proof}
Without loss of generality we set $b=1$. We will prove that when $|a| < 1$, the eigenvalues $\lambda_n$ are simple and
\[ |\lambda_n-n(n+1)| \le \frac{1}{n}  \]
for $n\ge 1$.
We require $|a| < 1$ for convenience, as it entails that all the Gershgorin discs have radius $< 1$, and so they are disjoint. The analysis is similar for an arbitrary $a$, but would not  apply to a finite number of the Gershgorin discs, namely those corresponding to $n(n+1)$ with $n\le |a|$.

We fix $n$ and consider the finite $(N+1)\times (N+1)$ tridiagonal matrix $T_N$ with the three entries on its $m$th row as in \eqref{eq:entries} and assume that $n<N$. By Gershgorin's theorem, $T_N$  has precisely one eigenvalue $\lambda_n^{(N)}$ in the disc centered at $n(n+1)$ of radius $a (\frac{n}{2n-1}+\frac{n+1}{2n+3})$. The corresponding eigenvector $(x_0^{(N)},\ldots , x_N^{(N)})$ must satisfy $|x_n^{(N)}|\ge |x_j^{(N)}|$ for $j=0,\dots,N$ since otherwise $\lambda_n$ would lie in another Gershgorin disc. We assume in the sequel that $x_n^{(N)}=1$.

Since the radii of the Gershgorin discs are $\le 1$, we see by considering the $(n-1)$-th and $(n+1)$-th rows of the relation $T_Nx^{(N)}=\lambda_n^{(N)} x^{(N)}$ that
\[ |x_{n\pm 1}^{(N)}|\le \frac{1}{n}, \]
and we may therefore conclude that
\[ |\lambda_n^{(N)}-n(n+1)|\le \frac{1}{n} , \quad n\ge 1, \]
by again applying Gershgorin's argument. By applying \eqref{eq:tridiag} inductively, we now get that
\begin{equation} \label{eq:fastdecay} \frac{|x_{n+2k+i}^{(N)}|}{\max(|x_{n+2k}^{(N)}|, |x_{n+2k-1}^{(N)}|)} \le \frac{1}{(4k+2)n+(2k+1)^2}, \quad k\ge 0, \ i=1,2. \end{equation}
Similarly we get
\[ \frac{|x_{n-2k-i}^{(N)}|}{\max(|x_{n-2k}^{(N)}|, |x_{n-2k+1}^{(N)}|)} \le \frac{1}{(4k+2)n-2k(2k-1)-1}, \quad 2k+i \le n,  \ i=1,2. \]
These uniform bounds allow us to apply a compactness argument to conclude that there exists a sequence $N_j$ such that
$\lambda_n^{(N_j)}$ converges and also that $x_m^{(N_j)}$ converges for every $m$. This means that we have identified an eigenvalue and an eigenvector of the infinite matrix~$T$.

Note that by \eqref{eq:fastdecay}, the coefficients of the eigenfunction decay super-exponentially, so $L_a$ does indeed map every eigenfunction into $L^2(-1,1)$, and even $C^{\infty}[-1,1]$.
\end{proof}

For our numerical computations to be described in~\S~\ref{sec:numerics}, we will need the following additional information.
\begin{lemma} \label{lem:specimproved}
For $0<a<3/2$, the spectrum of $\mathcal{L}_{a,1}$ is real and simple. 
Moreover, for any $3/2>\varepsilon>0$ the eigenvalues and eigenvectors depend continuosly on $a\in[0,3/2-\varepsilon]$ and are exponentially well approximated by the eigenvalues and eigenvectors of the truncated system $T_N$ as $N\to \infty$.
\end{lemma}
\begin{proof}
	We first note that if $a>0$ and $\lambda$ is a real eigenvalue of the matrix $T_N$, then $\lambda>0$ and the eigenvector satisfies $\xi_n>0$, $n=0,\dots,N$. We see this by solving the equations~\eqref{eq:tridiag} from two sides: On one hand, if $\lambda\le k(k+1)$, then we show inductively that $\xi_N, \xi_{N-1},\dots,\xi_{k-1}$ have the same sign, and on the other hand, if $\lambda\ge \ell(\ell+1)$, then $\xi_0,\dots,\xi_{\ell+1}$ have the same sign. Thus taking $\ell=k-1$ for some $k$, we find  that the sign of $\xi_n$ is constant. Since $\xi_n>0$, Gershgorin's argument can be improved to two-sided estimates
		\[-\frac{k}{2k-1}a \le \lambda_k - k(k+1) \le \frac{k+1}{2k+3}a\]
	when $a$ is sufficiently small, since in that case $\lambda_k$ is guaranteed 
to be real. Finally, for $a<3/2$, the above intervals do not intersect, so a 
homotopy argument shows that the spectrum remains real and simple.
	
We will now see that the eigenvectors and eigenvalues of the truncated
matrix $T_N$ converge exponentially as $N\to\infty$. Take an eigenvector $v$ of 
$T$ and its corresponding eigenvalue $\lambda$. Consider 
the vector $v_N$ in $\R^{N+1}$ which has the same first $N+1$ coordinates as~$v$. We 
have $T_N v_N = \lambda v_N+ r_N$, where $r_N$ is an error term such that 
$\|r_N\|\le c^{-N}$ for some $c>1$ because the coefficients of $v$ decay super-exponentially. 
By the Bauer--Fike theorem (see \cite{BF60}), there is an eigenvalue $\lambda_k$ of 
$T_N$ such that 
\[
 |\lambda_k - \lambda| \le \kappa(V_N) \frac{\|r_N\|}{\|v_N\|}, 
\]
where $V_N$ is the matrix of eigenvectors of $T_N$ and $\kappa(V_N) \coloneqq  \|V_N\| 
\|V_N^{-1}\|$. The norm of the matrix~$V_N$ grows polynomially in $N$ 
(uniformly for $a$ in $[0,3/2-\varepsilon]$ and $\varepsilon>0$), and similarly the norm of $V_N^{-1}$ grows polynomially, which can be obtained 
from the left eigenvectors of $T_N$. Thus altogether we have exponential 
convergence (uniformly in $a$) of the eigenvalue $\mu_N(a)$ 
towards $\lambda(a)$, and therefore also continuous dependence on $a$.

As for the corresponding eigenvector, we will use that by the Gershgorin 
bounds obtained above, the eigenvalues of $T_N$ are well separated, i.e.,  
there is a $\delta>0$, independent of $N$, such that $|\lambda-\mu|\ge \delta$ 
for any two different eigenvalues of $T_N$ and any $a$ in $[0,3/2-\varepsilon]$.

To approximate eigenvectors, we make the expansion $v_N = \sum_{j} c_j u_j$, where $u_j$ are the eigenvectors of $T_N$ and get thus $T_N v_N=\sum_j c_j  \lambda_j u_j$. We may assume that $\lambda$ is close to some $\lambda_k$. Since $T_N v_N = \lambda 
v_N + r_N$, we have $\sum c_j(\lambda_j -\lambda) u_j = r_N$. For $k\ne j$ we have
that $|c_j| = \frac{|\langle r_N, w_j\rangle|}{|\lambda_j-\lambda|}$, where 
$w_j$ are the biorthogonal vectors associated to the right eigenvectors, i.e. 
the left eigenvectors of $T_N$. Finally, $|c_j|$ decays (uniformly) exponentially in $N$ for 
$k\ne j$.
\end{proof}

We close this section by recording a puzzling property of the tridiagonal matrix appearing in the proof of Theorem~\ref{thm:spectrum}. Arguing as in the proof of Corollary~\ref{cor:thirdorder}, we see that for a $\lambda$-eigenfunction $f$ of $\mathcal{L}_{a,b}$, the function $F(z)=f(z)f(-z)$ satisfies the third order differential equation
	\[(z^2F)''' +
	\Big(4b^2z^2-4\lambda-\frac{a^2}{z^2}\Big)F' +
	\Big(8b^2z-\frac{4\lambda}{z}\Big)F= 0\,.\]
Writing $F(z)=\sum_{n\ge0}u_nz^{2n}$, we get the recursion
	\[u_{n}(2n+2)(2n+1)(2n)+8nb^2u_{n-1}-a^2(2n+2)u_{n+1}=4\lambda u_{n}(1+2n),\]
which can be rewritten as
	\[\lambda u_n = n(n+1)u_n+2b^2\frac{n}{2n+1}u_{n-1}-\frac{a^2}{2}\frac{n+1}{2n+1}u_{n+1}.\]
Surprisingly, we recognize on the right-hand side the transpose of the tridiagonal matrix~$T$. In particular, setting $b=1$ and $u_{n}=\frac{\xi_n(-2/a)^n}{2n+1}$, we get 
	\[\lambda \xi_{n} = n(n+1)\xi_{n}-a\big(\frac{n}{2n-1}\xi_{n-1}-\frac{n+1}{2n+3}\xi_{n+1}\big),\]
which is exactly~\eqref{eq:tridiag}. Thus we see that $\widehat{f}(\xi)=c\sum_{n\ge0}(2n+1)(-ia/2)^n u_nP_n(\xi)$, $\xi\in(-1,1)$, for some constant $c$ (in fact, $c=\pi^{-1}$, since $\int_{-1}^{1}P_n(\xi)d\xi=2\delta_{n,0}$).

\section{The functional equation of Theorem~\ref{thm:main}}\label{sec:functionaleq}
It remains to establish the second part of Theorem~\ref{thm:main}, namely the claim that the solutions of the functional equation \eqref{eq:functional} are complex scalars times  $\Phi$. We split the proof into two parts. 

\begin{proof}[Proof that $\Phi$ satisfies \eqref{eq:functional}] By~\eqref{eq:diffbasic} that we already proved in~\S\ref{sec:ode}, we see that $\Phi$ is an eigenfunction of $\mathcal{L}_{a,b}$ for $a=\frac{1}{2\Cone}$, $b=\frac{\pi}{2}$, with eigenvalue $\lambda=-\frac{L_{\tau}(1)}{2\Cone}$. Therefore $\Phi$ satisfies~\eqref{eq:feqgeneral} for some values $\kappa_{+}$ and $\kappa_{-}$. Since $\Phi$ is real-valued on $\R$, we see that $\kappa_{-}=\overline{\kappa_{+}}$ by taking complex conjugation on both sides of~\eqref{eq:feqgeneral}. Setting $z=\tau_{2n-1}$, we find that
	$\kappa_{+}e^{-i\pi\tau_{2n-1}/2}+\kappa_{-}e^{i\pi\tau_{2n-1}/2}\to 0$ as $n\to\infty$,
	and since $\tau_{2n-1}=2n-1/2+o(1)$, this is equivalent to
	\[\kappa_{+}e^{i\pi/4}+\kappa_{-}e^{-i\pi/4} = 0.\]
	Next, since $\Phi(0)=1$ and positive zeros of $\Phi$ are $\{\tau_{2n-1}\}_{n\ge1}$, the number $\Phi(\tau_{2n})$ has sign $(-1)^n$ for large $n$. Setting $z=\tau_{2n}$ in~\eqref{eq:feqgeneral} and taking limits, we get
	\[\kappa_{+}e^{-i\pi/4}+\kappa_{-}e^{i\pi/4} > 0,\]
	so that $\kappa_{\pm}=e^{\pm i\pi/4}\kappa$ for some $\kappa>0$.  Plugging these values into~\eqref{eq:involution}, we get $\kappa^2=\frac{a}{4b}$, and so $\kappa=\frac{1}{\sqrt{4\pi \Cone}}$. With these computations, we obtain precisely~\eqref{eq:functional} by dividing~\eqref{eq:feqgeneral} by $z$. \end{proof}

\begin{proof}[Proof that all solutions to \eqref{eq:functional} are  constant multiples of $\Phi$] We assume to the contrary that $F$ is a nontrivial entire function satisfying
\[ F(z) e^{\frac{1}{4\mathscr C z}}= \frac{e^{i\frac{\pi}{2}(z-\frac12)} F\big(\frac{1}{2\pi i \mathscr{C} z}\big)+e^{-i\frac{\pi}{2}(z-\frac12)} F\big(-\frac{1}{2\pi i \mathscr{C} z}\big)}{2\sqrt{\pi \mathscr{C}}z} ,\qquad z\in \C^\ast,\]
with $F$ not being a multiple of $\Phi$. We will show that this assumption leads to a contradiction.

We may assume that $F$ is a real entire function satisfying $F(0)=0$. This follows from the observation that also $F-F(0)\Phi$ and $F+F^\ast$ (with $F^\ast(z)\coloneqq \overline{F(\overline{z})}$) satisfy the functional equation. 
We start from the fact that the function
\[ G_{\varepsilon}(z)\coloneqq \Phi(z)+\varepsilon F(z)  \]
also satisfies the functional equation for every complex $\varepsilon$. The functional equation entails that $F$ must be of exponential type $\frac{\pi}{2}$ and that
\begin{equation} \label{eq:Fasymp} F(z)= O\Big(\frac{e^{\frac{\pi}{2}|y|}}{|z|^2} \Big), \quad z\to \infty . \end{equation} It follows from Rouch\'{e}'s theorem and \eqref{eq:Fasymp} that we may write
\[ G_{\varepsilon}(z)=(1+\varepsilon F(0)) \prod_{n=1}^\infty \Big(1+(-1)^n\frac{z}{t_n(\varepsilon)}\Big), \]
with $t_n(\varepsilon)\to \tau_n$ when $\varepsilon\to 0$, uniformly in $n$. We will henceforth assume that $\varepsilon$ is real and small enough, so that $(t_n(\varepsilon))$ is a sequence of real numbers. 

Now the functional equation yields the system of equations
\[ e^{i\frac{\pi}{2}(-1)^nd_n(\varepsilon)} F\big((-1)^{n+1} \frac{1}{2\pi i \mathscr{C}t_n(\varepsilon) }\big)-e^{-i\frac{\pi}{2}(-1)^nd_n(\varepsilon)} F\big((-1)^n\frac{1}{2\pi i \mathscr{C} t_n(\varepsilon)}\big)=0, \]
where $d_n(\varepsilon)\coloneqq n+\frac12-t_n(\varepsilon)$.
Hence
\[ i\pi (-1)^n d_n(\varepsilon) = \log F\big((-1)^n\frac{1}{2\pi i \mathscr{C} t_n(\varepsilon)}\big) -  \log F\big((-1)^{n+1} \frac{1}{2\pi i \mathscr{C}t_n(\varepsilon) }\big) . \]
It follows that for $\varepsilon$ small enough, $(d_n(\varepsilon))$ is a fixed point of the mapping $K$ defined by the formula
\begin{align*} (K\xi)_n\coloneqq \frac{(-1)^n}{\pi i} \sum_{m=1}^\infty  \Big(& \log\Big(1+\frac{(-1)^{n+m}}{2\pi i \mathscr{C}(n+\frac12-\xi_n)(m+\frac12-\xi_m)}\Big)\\
	& -\log\Big(1+\frac{(-1)^{n+m+1}}{2\pi i \mathscr{C}(n+\frac12-\xi_n)(m+\frac12-\xi_m)}\Big)\Big).\end{align*}
This means that we have
\begin{equation} \label{eq:fixed} Kd(\varepsilon)-K\delta   =d(\varepsilon)-\delta, \end{equation}
where $\delta_n\coloneqq d_n(0)=n+\frac12-\tau_n$. On the other hand, by the mean value theorem,
\begin{equation} \label{eq:meanvalue} (Kd(\varepsilon))_n-(K\delta)_n=\nabla (K\xi)_n \big|_{\xi=c\delta+(1-c)d(\varepsilon)} \cdot \big(d(\varepsilon)-\delta\big) \end{equation}
for some $0<c<1$, where the dot signifies inner product between real sequences. Since $d(\varepsilon)\neq \delta$ for all small $\varepsilon$, \eqref{eq:meanvalue} will contradict \eqref{eq:fixed} if we can show that the $\ell^1$ norms of the gradients in \eqref{eq:meanvalue} are uniformly $<1$ when $\varepsilon$ is small enough. 

Since $t_n(\varepsilon) \to \tau_n$ uniformly in $n$, it suffices to estimate the gradients at $\varepsilon=0$. We compute and get
\begin{align*} \frac{\partial (K\xi)_n}{\partial \xi_n}\big|_{\xi=\delta} & = (-1)^{n+1} \Bigg( \frac{4 \mathscr{C} \tau_n }{2\pi i \mathscr{C} \tau_n^2+1} - \frac{4 \mathscr{C} \tau_n }{2\pi i \mathscr{C} \tau_n^2-1}\\
	& + \sum_{m\neq n} \Big( \frac{2 \mathscr{C} \tau_m }{2\pi i \mathscr{C} \tau_n\tau_m+(-1)^{n+m}} - \frac{2 \mathscr{C} \tau_m }{2\pi i \mathscr{C} \tau_n\tau_m-(-1)^{n+m}}\Big)\Bigg) \\
	& = 4\mathscr{C}(-1)^{n+1} \Bigg( \frac{(-1)^n \tau_n }{4\pi^2 \mathscr{C}^2 \tau_n^4+1}
	+ \sum_{m=1}^\infty \frac{(-1)^m \tau_m }{4\pi^2 \mathscr{C}^2 \tau_n^2\tau_m^2+1} \Bigg).
\end{align*}
By a similar computation, we find that when $m\neq n$,
\[ \Big|\frac{\partial (K\xi)_n}{\partial \xi_m}\big|_{\xi=\delta}\Big| = \frac{8 \mathscr{C}\tau_n}{4\pi^2 \mathscr{C}^2\tau_n^2\tau_m^2+1}    . \]
Hence
\[ \big\| \nabla (K \xi)_n \big|_{\xi=\delta} \big\|_{\ell^1} \le \frac{16 \mathscr{C}\tau_1}{4\pi^2 \mathscr{C}^2 \tau_1^4+1} +\sum_{m=2}^\infty  \frac{8 \mathscr{C}\tau_1}{4\pi^2 \mathscr{C}^2\tau_1^2\tau_m^2+1}, \]
which is $<0.5$ by a crude estimation. We have thus obtained the desired conflict between \eqref{eq:meanvalue} and \eqref{eq:fixed}.
 \end{proof}

The preceding proof shows that the linear space of
entire functions that satisfy \eqref{eq:feqgeneral} is one-dimensional in the special case when $a=\frac{1}{2\Cone}$, $b=\frac{\pi}{2}$, and $\kappa_{\pm}=\frac{e^{\pm i\pi/4}}{\sqrt{4\pi \Cone}}$. 
It would be desirable to find an alternate proof that would allow us to conclude that this linear space is at most one-dimensional for general $a$, $b$, and $\kappa_\pm$ satisfying the admissibility condition  $\kappa_+^2-\kappa_-^2=\frac{ia}{2b}$. Since this solution space is invariant under $\mathcal{L}_{a,b}$ by the commutation relation of Lemma~\ref{lem:commrel}, we would then be able to infer that all entire solutions of \eqref{eq:feqgeneral} would also be eigenfunctions of $\mathcal{L}_{a,b}$.

\section{Power series expansion of \texorpdfstring{$\tau_n$}{tau}}
\label{sec:taun}
We will now use the functional equation \eqref{eq:functional} to establish the power series expansion of~$\tau_n$ proclaimed in Theorem~\ref{thm:zeros}.
We begin by observing that \eqref{eq:functional} along with the fact that
	\[ \arg \Phi(iy)=\sum_{n=1}^\infty (-1)^n \arctan \frac{y}{\tau_n}, \]
	implies that
\begin{equation} \label{eq:zeros} n+\frac12-\tau_n=\frac{2}{\pi} \sum_{m=1}^\infty
(-1)^{m+1} \arctan \frac{1}{2\pi \Cone \tau_n \tau_m}, \quad n\ge 1 . \end{equation}

Now using the expansion
\[ \arctan z = \sum_{m=1}^\infty (-1)^{m+1} \frac{z^{2m-1}}{2m-1}, \]
we may rewrite \eqref{eq:zeros} as
\[ n+\frac12-\tau_n = \frac{2}{\pi} \sum_{m=1}^\infty
(-1)^{m}  \frac{L_\tau(2m-1)}{(2m-1)(2\pi \Cone \tau_n)^{2m-1}}, \quad n\ge 1. \]
Introducing next the auxiliary series
	\[g(z)\coloneqq \frac{1}{\pi i} \sum_{m\ge0}\frac{2L_{\tau}(2m-1)z^{2m-1}}{2m-1}\,,\qquad 0<|z|<\tau_1\, ,\]
we see that the latter identity can be written as
	\[g\Big(\frac{1}{2\pi i \Cone \tau_n}\Big) = n+\frac12, \quad n = 1, 2, \ldots\, .\]
We define a power series $\rho(z)=\sum_{m\ge1}a_mz^m$ by the implicit equation
	\begin{equation}
	g\Big(\frac{1}{2\pi i \Cone (1/z-\rho(z))}\Big) = \frac{1}{z}\,,\qquad |z|<\varepsilon.
 	\end{equation}
Indeed, this equation simply asserts that $z\mapsto1/g(z)$ is inverse to $z\mapsto \frac{1}{2\pi i \Cone (1/z-\rho(z))}$ in a small neighborhood of the origin, so $\rho$ is well-defined and is an odd function. In what follows, given a function~$f$ analytic in a small punctured neighborhood of the origin, we denote by $[z^n]f(z)$ the coefficient of $z^n$ in the Laurent series of $f$ at $0$. 
\begin{theorem}\label{thm:zeroexp}
	We have $\rho(z)=\sum_{m\ge1}a_mz^m$, where
	\[a_m = [z]\frac{g(z)^m}{2\pi i \Cone m}\quad \text{and} \quad a_m\ge 0,  \quad m\ge1.\]
	This power series converges if and only if $|z|\le 2$. \end{theorem}
It is immediate from this theorem and the preceding discussion that
we now have the power series representation
	\begin{equation*}\label{eq:analyticzeros}
	\tau_n = n+1/2-\rho\Big(\frac{1}{n+1/2}\Big) = n+1/2-\sum_{m\ge1}a_{2m-1}(n+1/2)^{1-2m}
	\end{equation*}
	for all $n\ge1$, and so we obtain Theorem~\ref{thm:zeros} as a corollary. 
Note that $\rho$ is an odd function, so $a_m$ vanish for even $m$. The first few nontrivial relations between $a_m$ and $L_{\tau}(2m-1)$ are
	\begin{equation} \label{eq:a1q} \pi^2a_1 = -\frac{L_{\tau}(1)}{\Cone}\,,\ \
	  \pi^4a_3 = \frac{L_{\tau}(1)^2}{\Cone^2}+\frac{L_{\tau}(3)}{12\Cone^3}\,,\ \
	  \pi^6a_5 = -2\frac{L_{\tau}(1)^3}{\Cone^3}-\frac{L_{\tau}(3)L_{\tau}(1)}{3\Cone^4}-\frac{L_{\tau}(5)}{80\Cone^5}\,.\end{equation} 
From~\eqref{eq:charphi} it follows that $L_{\tau}(2k-1)$ is an element in $\Q[\Cone,L_{\tau}(1),\pi]$. For instance,
	\begin{align*}
	L_{\tau}(3) &= 24L_{\tau}(1)\Cone^2 + (\pi^2/2+2L_{\tau}(1)^2)\Cone\,,\\
	L_{\tau}(5) &= 1920L_{\tau}(1)\Cone^4 + (40\pi^2 + 448L_{\tau}(1)^2)\Cone^3 + (2L_{\tau}(1)\pi^2 + 8L_{\tau}(1)^3)\Cone^2.
	\end{align*}
	 Therefore, the coefficients $a_m$ can also be written in this way:
	 \begin{equation} \label{eq:allaq}
	  \begin{aligned} 
	 \pi^4a_3 &= \frac{48L_{\tau}(1)\Cone + \pi^2 + 28L_{\tau}(1)^2}{24\Cone^2}\,,\\
	 \pi^6a_5 &= -\frac{2880L_{\tau}(1)\Cone^2 + (60\pi^2 + 1632L_{\tau}(1)^2)\Cone + (23L_{\tau}(1)\pi^2 + 332L_{\tau}(1)^3)}{120\Cone^3}\,,
	 \end{aligned}
	 \end{equation}
	 and so on.
\begin{proof}[Proof of Theorem~\ref{thm:zeroexp}]
	The first claim is simply a form of Lagrange's inversion theorem, and can be seen directly from the following residue calculation:
	\begin{multline*}
	a_m = [z^m]\rho(z) = -\mathrm{res}_{z=0}\frac{1/z-\rho(z)}{z^{m+1}}dz
	= -\mathrm{res}_{w=0}\frac{g(w)-\rho(1/g(w))}{g(w)^{-m-1}}d(1/g(w))\\
	= \frac{1}{2\pi i \Cone m}\mathrm{res}_{w=0}w^{-1}d(g(w)^{m})
	= \frac{1}{2\pi i \Cone m}\mathrm{res}_{w=0}w^{-2}g(w)^{m}dw = [z]\frac{g(z)^m}{2\pi i \Cone m}\,.
	\end{multline*}
	As an immediate consequence, we see that $\rho(u)$ has the representation
	\begin{equation} \label{eq:rhorep} \rho(u) = -\frac{1}{(2\pi i)^2 \Cone}\oint_{\gamma} \log(1-ug(z))\frac{dz}{z^2}\,\end{equation}
	for small values of $u$, where $\gamma$ is any path around zero on which $|g(z)|$ is bounded from above. In view of \eqref{eq:rhorep},  it suffices to find a path $\gamma_\varepsilon$ around $0$ along which $|g(z)|\le \frac12 + \varepsilon $ to prove that the radius of convergence $R$ of $\sum_{m\ge1}a_mz^m$ is at least $2$. By symmetry, it will be enough to prescribe this path as we move in the first quadrant, from some point on the positive real axis to a point on the positive imaginary axis. We will use the identity
	\begin{equation} \label{eq:loggg}  \exp(\pi i g(z)) = \frac{A(z)}{A(-z)}\,,\qquad 0<|z|<\tau_1\,,\end{equation}
	 which defines $g(z)$ in the punctured strip
$\{x+iy:\ -\tau_1< x<\tau_1 \}\setminus \{0\}$. By \eqref{eq:functional}, we have
\[ \frac{A(z)}{A(-z)}= -i \left(\frac{1-i w}{1+iw}\right), \]
where
\begin{equation} \label{eq:wdef} w(z)\coloneqq e^{i\pi z} \frac{\Phi\big(\frac{1}{2\pi i \Cone z}\big)}{\Phi\big(-\frac{1}{2\pi i \Cone z}\big)} . \end{equation}
We see that $|w(z)|<1$ when $z=x+iy$ satisfies $x\ge 0$ and $y> 0$, whence $\frac{1-i w(z)}{1+iw(z)}$ lies in the right half-plane. Since $g(x)$ takes imaginary values for real $x$, we find from \eqref{eq:loggg} that
$\operatorname{Im} \pi i g(z) $ lies in the lower half-plane.

We now prove that there exists a curve $\gamma^+$ in the first quadrant along which $\operatorname{Re} w(z) <0$ and $\operatorname{Im} w(z) =0$, with a parameterization of the form
\[ \gamma^+(y)=\Xi(y)+iy, \quad  \frac12 < \Xi(y) < 1, \ y\ge 0.\]
To this end, we begin by noting that for $\frac12 \le x \le 1$ and $y\ge 0$, the sequence
\begin{align*} \kappa_n\coloneqq  \arg & \left(1+(-1)^n\frac{1}{\tau_n 2\pi i \Cone(x+iy)}\right) -  \arg \left(1-(-1)^n\frac{1}{\tau_n 2\pi i \Cone(x+iy)}\right) \\
& =  \arg  \left(1-(-1)^n\frac{y+ix}{\tau_n 2\pi  \Cone(x^2+y^2)}\right) -  \arg \left(1+(-1)^n\frac{y+ix}{\tau_n 2\pi \Cone(x^2+y^2)}\right) \end{align*}
is alternating with $0<\kappa_1<\frac{\pi}{2}$ and $|\kappa_n|\searrow 0$. Hence
\[ 0<\arg \frac{A(x+iy)}{A(-x-iy)}<\frac{\pi}{2}, \quad \frac12  \le x \le 1, \ y\ge 0. \]
In view of  \eqref{eq:wdef}, this means that for every $y\ge 0$, there exists a number $\Xi(y)$ between $\frac12$ and $1$ such that $\arg w\big(\Xi(y)+iy\big)=\pi$. We have thus verified the existence of the curve $\gamma^+$.

 We see that $g(z)$ is real for $z$ in $\gamma^+$ and in fact
\begin{equation} \label{eq:negative} - \frac{1}{2} \le g(z) < 0, \quad z\in \gamma^+. \end{equation}
To finish the proof that $R\ge 2$, we choose a sufficiently large $Y$ such that $|g(x+iY)|\le \frac12+\varepsilon$ for $x+iY$ with $0 \le x \le \Xi(Y)$, and we let $\gamma_{\varepsilon}$ intersected with the first quadrant be the curve parametrized as $\Xi(y)+iy$ for $0\le y < Y$ and $x+iY$ for $0\le x \le x(Y)$.

To prove that $R=2$, we assume that $R>2$ and show that this leads to a contradiction. First, since $g(-iy)\to \frac12$ when $y\to \infty$, we have that $\rho(x)\to \frac12$ when $x\to 2$. Then by our assumption that $R>2$,
\[ \lim_{x\to 2} \frac{\rho(x)-\frac12}{x-2} \]
exists and takes a finite value, and this is equivalent to the assertion that
\[ \lim_{x\to 2} \frac{\rho(x)-\frac1x}{x-2} \]
exists and takes a finite value. But since $z\mapsto\frac{1}{g(z)}$ is inverse to $z\mapsto \frac{1}{2\pi i \Cone (1/z-\rho(z))}$, we have
\[  \lim_{x\to 2} \frac{\rho(x)-\frac1x}{x-2} = \lim_{y\to \infty} \frac{y}{2\pi  \Cone (\frac{1}{g(-iy)}-2)}. \]
Here the limit to the right cannot be finite because $g(-iy)=\frac12+O(e^{-\pi y})$, and so our assumption that $R>2$ cannot be true.

The preceding argument shows that
\begin{equation}  \label{eq:oncirc} \rho(x)-\frac12 =O\Big(\frac{1}{\log \frac{1}{2-x}}\Big) \end{equation}
when $x\to 2^{-}$, and so $\rho(2)=\frac12$. Absolute convergence of the power series on its radius of convergence will therefore follow once we have shown that $a_m\ge 0$ for all $m\ge 1$.
To this end, we note that it suffices to consider the case when $m$ is odd, since $\rho$ is an odd function. We start from the formula
\[ a_m=-\frac{1}{(2\pi )^2 \Cone m}\oint_{\gamma}\, [g(z)]^m \frac{dz}{z^2}, \]
where $\gamma$ is again a suitable path around zero. We let $\gamma^+$ be as above and find that, by symmetry,
\[ a_m=-  \frac{4}{(2\pi )^2 \Cone m} \operatorname{Re} \int_{\gamma^+}\,  [g(z)]^m \frac{dz}{z^2},\]
which can be expressed as
 \[ a_m=-  \frac{4}{(2\pi )^2 \Cone m}  \int_{0}^\infty \,  [g(\Xi(y)+iy)]^m \big(2y \Xi(y)+\Xi'(y)(\Xi^2(y)-y^2)\big) \frac{dy}{|\Xi(y)+iy|^4}.\]
 Since $ [g(\Xi(y)+iy)]^m<0$ by \eqref{eq:negative}, we are done if we can show that
 \begin{equation} \label{eq:keypos} \frac{\Xi'(y)}{\Xi(y)} (y^2-\Xi^2(y)) \le 2 y \end{equation}
 holds for all $y>0$. To prove \eqref{eq:keypos}, we start from the equation
 \[ \arg w(\Xi(y)+iy)=\pi , \]
 which by implicit differentiation and \eqref{eq:wdef} yields
 \[ \pi \Xi'(y) =\operatorname{Im} \left( \frac{\Phi'\big(\frac{1}{2\pi i \Cone z}\big)}{\Phi\big(\frac{1}{2\pi i \Cone z}\big)}
 + \frac{\Phi'\big(-\frac{1}{2\pi i \Cone z}\big)}{\Phi\big(-\frac{1}{2\pi i \Cone z}\big)}\right) \frac{(\Xi'(y)+i)}{2\pi i \Cone(\Xi(y)+iy)^2} .\]
Using the definition of $\Phi$, we find that
 \[ \frac{\Phi'\big(\frac{1}{2\pi i \Cone z}\big)}{\Phi\big(\frac{1}{2\pi i \Cone z}\big)}
 + \frac{\Phi'\big(-\frac{1}{2\pi i \Cone z}\big)}{\Phi\big(-\frac{1}{2\pi i \Cone z}\big)}
 = \sum_{n=1}^{\infty} \Big(\frac{1}{\frac{1}{ 2\pi i \Cone(x+iy)}+(-1)^n\tau_n} + \frac{1}{\frac{-1}{ 2\pi i \Cone(x+iy)}+(-1)^n\tau_n}\Big), \]
 and so
 \[ \pi \Xi'(y)=- \im  \sum_{n=1}^{\infty} \frac{4\pi i \Cone (-1)^n (\Xi'(y)+i)}{(\tau_n^2  4\pi^2 \Cone^2(\Xi(y)+iy)^2+1)} .\]
 Hence
 \[ \frac{\Xi'(y)}{\Xi(y)}= \sum_{n=1}^\infty \frac{32 (-1)^{n+1}\tau_n^2 \pi^2 \Cone^3  y}{|\tau_n^2  4\pi^2 \Cone^2(\Xi(y)+iy)^2+1|^2} \cdot \left(1+\sum_{n=1}^\infty \frac{4(-1)^n \Cone\big(4\tau_n^2 \pi^2 \Cone^2(\Xi^2(y)-y^2)+1\big)}{|\tau_n^2  4\pi^2 \Cone^2(\Xi(y)+iy)^2+1|^2} \right)^{-1}.\]
 We see from this expression that $\Xi'(y)> 0$ for all $y>0 $, and so \eqref{eq:keypos} holds trivially when $0<y\le \Xi(y)$. Furthermore, when $y\ge \Xi(y)$, the same expression yields
\begin{align*}   \frac{\Xi'(y)}{\Xi(y)} \le & \sum_{n=1}^\infty \frac{32 (-1)^{n+1}\tau_n^2 \pi^2 \Cone^3  y}{|\tau_n^2  4\pi^2 \Cone^2(\Xi(y)+iy)^2+1|^2} \le \frac{32 \tau_1^2 \pi^2 \Cone^3  y}{|\tau_1^2  4\pi^2 \Cone^2(\Xi(y)+iy)^2+1|^2} \\
& \le \frac{1}{2  \tau_1^2 \pi^2 \Cone \Xi^2(y) y} \le \frac{2}{ \tau_1^2 \pi^2 \Cone  y} ,  \end{align*}
where we in the last step used that $\Xi(y)>\frac12$. Since clearly $ \tau_1^2 \pi^2 \Cone\ge 1$ ($\tau_1>1$ by the main result of~\cite{BORS} and $\Cone>1/2$ by~\eqref{eq:hb}), we conclude that \eqref{eq:keypos} holds for all $y>0$.
\end{proof}

\section{Behavior of the Fourier transform of \texorpdfstring{$\phi$}{phi}}
\label{sec:phihat}
We now turn to the proof of Theorem~\ref{thm:phifourier}.
We start with a lemma that will allow us to compute the Fourier transform of eigenfunctions of the differential operators $\mathcal{L}_{a,b}$ from~\S~\ref{sec:family}.

\begin{lemma} \label{lem:genfourier}
	Let $b>0$ and $f,g,h$ be three entire functions satisfying
	\[zf(z) = e^{ibz}g(1/z) + e^{-ibz}h(1/z).\]
	Then $\widehat{f}$ is supported in $[-b,b]$, and on $[-b,b]$ it is an analytic function whose Taylor expansions at the endpoints are given by
	\[\widehat{f}(\xi) =   2\pi i\sum_{n=0}^{\infty}\frac{g^{(n)}(0)}{n!^2}(  i(b-\xi))^n
	                            = -2\pi i\sum_{n=0}^{\infty}\frac{h^{(n)}(0)}{n!^2}(-i(b+\xi))^n\,.\]
\end{lemma}
\begin{proof}
	Let $\gamma_{+}$ be a path in $\C$ going from $-\infty$ to $-1$, then from $-1$ to $1$ along a semi-circle in the upper half-plane, and finally from $1$ to $\infty$, and let $\gamma_{-}$ be a similar path going along a semi-circle in the lower half-plane instead. A standard contour integral calculation shows that for any entire function $F$ with $F(0)=0$ we have
	\[\int_{\gamma_{\pm}}e^{i\lambda z}F(1/z)dz = \begin{cases}
		0,\qquad\qquad\qquad\qquad\qquad\quad \pm\lambda>0,\\
		\mp2\pi i\,\mathrm{res}_{z=0}e^{i\lambda z}F(1/z),\qquad \pm\lambda<0.
	\end{cases}\]
	(If $F'(0)\ne0$, then the integral on the left should be taken as a limit over symmetric truncations of $\gamma$.)
	Since $\widehat{f}(\xi) = \int_{\gamma_{\pm}}f(x)e^{-i\xi x}dx=\int_{\gamma_{\pm}}(e^{ibx}g(1/x)+e^{-ibx}h(1/x))e^{-i\xi x}\frac{dx}{x}$, applying the above observation shows that $\widehat{f}$ is supported in $[-b,b]$, and for $\xi$ in $(-b,b)$ we have
	\[\widehat{f}(\xi) 
	= -2\pi i\,\mathrm{res}_{z=0}e^{-i(b+\xi)z}z^{-1}h(1/z)
	= 2\pi i\,\mathrm{res}_{z=0}e^{i(b-\xi)z}z^{-1}g(1/z),\]
	which implies the claim upon computing the residues.
\end{proof}
Applying this lemma to~\eqref{eq:feqgeneral}, we obtain the following.

\begin{corollary} \label{cor:fourier}
	For any $b>0$ and any eigenfunction $f$ of $\mathcal{L}_{a,b}$, the Fourier transform of $f$ is supported in $[-b,b]$, and is equal to a restriction of an analytic function. Moreover, the numbers $\kappa_{\pm}$ can be computed from $\kappa_{\pm} = \pm 2\pi i\widehat{f}(\mp b)$.
\end{corollary}

We also get the desired regularity of the Fourier transform of the H\"ormander--Bernhardsson function.

\begin{proof}[Proof of Theorem~\ref{thm:phifourier}]
Multiplying~\eqref{eq:functional} by the same identity after substitution
$z\mapsto-z$ and using the fact that $\phi(z)=\Phi(z)\Phi(-z)$, we get
\begin{equation} \label{eq:phifunctional}
	-4\pi \Cone z^2\phi(z) =
	e^{\pi i z}\Phi^2\Big(\frac{1}{2\pi i \Cone z}\Big)+e^{-\pi i
		z}\Phi^2\Big(-\frac{1}{2\pi i \Cone z}\Big)\,.
\end{equation}
By Lemma~\ref{lem:genfourier},
	\[-\widehat{\phi}(\pi\xi) = 2\pi i\,\mathrm{res}_{z=0}\frac{e^{\pi i(1-\xi)z}}{4\pi \Cone z^2}\Phi^2\Big(\frac{1}{2\pi i\Cone z}\Big),\]
which after a short calculation implies~\eqref{eq:phifourier}.
\end{proof}

\section{Summation formulas}
\label{sec:summation}
We now return to $\eqref{eq:dercond}$ which we reformulate as a summation 
formula valid for all $f$ in $PW^1$ in terms of the values at the zeros 
$\pm \tau_n$ of $\phi$:
\[
-\frac{f'(0)}{2\Cone} = \sum_{n= 1}^\infty (-1)^n(f(\tau_n)-f(-\tau_n)).
\]
The observation to be made in this section is that any eigenfunction of the differential operator
$\mathcal{L}_{a,b}$ defined in  \eqref{eq:diffspecial} yields an
analogous summation formula. We set for convenience $b= \pi/2$ so that all formulas pertain to functions of exponential type at most $\pi$. For an entire function $g$, we let $\mathcal{Z}(g)$ denote its zero set, with multiplicites accounted for in the usual way. 
\begin{theorem}\label{thm:representation}
	Let $a$ be a nonzero complex number and $g$ an eigenfunction of 
	$\mathcal{L}_{a,\frac{\pi}2}$. Then for every 
	$f$ in $PW^1$ we have
	\begin{equation} \label{eq:summ}
	a f'(0) = \sum_{\mu \in \mathcal{Z}(g)} \big(f(\mu)-f(-\mu)\big).
	\end{equation}
	\end{theorem}
We may take note of the following subtlety at this point. If $a$ were positive and 
\[ g(z)g(-z)=\prod_{n=1}^\infty \left(1-\frac{z^2}{t_n^2}\right) \] 
with $0<t_1<t_2 < \cdots $,  we may deduce from \eqref{eq:summ} a counterpart to \eqref{eq:basicid} of the form
\begin{equation} \label{eq:L1ort}    2a f(0)+c\widehat{f}(0)=\int_{-t_1}^{t_1} f(x) dx + \sum_{n=1}^\infty (-1)^n \int_{t_n}^{t_{n+1}} (f(x)+f(-x)) dx \end{equation}
for some real constant $c$. We will then have $c\neq 0$ if and only if $a=\frac{1}{2\Cone}$ and $g=\Phi$. Indeed, if $c= 0$, \eqref{eq:L1ort} would imply that $g=\Phi$ by the characterization of 
$\varphi$ in terms of $L^1$~orthogonality (see \cite[Thm. 3.6]{BCOS}). The fortuitous fact used in the proof of Lemma~\ref{lem:summform} to show that \eqref{eq:summ} would imply $c=0$ for admissible sequences $(t_n)$, is that such sequences are at finite $\ell^2$ distance from the sign changes of $\sgn \cos(\pi x)$ which in the terminology of \cite{LS18} is a real extremal signature for $PW^1$, i.e., its Fourier transform vanishes on $(-\pi, \pi)$. We may conclude that for no other eigenfunction $g$ with $a>0$ and real simple zeros, is it possible to find a real extremal signature  for $PW^1$ with  
sign changes similarly close to $\pm \mathcal{Z}(g)$. 

The following observation allows us to connect eigenfunctions of 
$\mathcal{L}_{a,b}$ with the kind of summation formulas appearing in \S~\ref{sec:quaddiffeq}.

\begin{lemma}
	If $g(z)$ is an eigenfunction of $\mathcal{L}_{a,b}$, normalized by
	$g(0)=1$, then the function
	$A(z)=e^{\frac{a}{2z}}g(z)$ satisfies
	\begin{equation}\label{eq:quadratic2}
		A'(z)A(-z)+A'(-z)A(z) = -\frac{a}{z^2},\qquad z\in \C\setminus \{0\}.
	\end{equation}
\end{lemma}
\begin{proof}
	By~\eqref{eq:diffspecialalt}, the functions $A(z)$ and $A(-z)$ are linearly independent and both satisfy the differential equation
	\[z^2h''(z)+2zh'(z)+\Big(b^2z^2-\frac{a^2}{4z^2}\Big)h(z)=\lambda h(z)\]
	Therefore their Wronskian is a constant multiple of $z^{-2}$. The value of 
	the constant can be determined by a direct calculation using the fact that 
	$g(0)=1$.
\end{proof}
\begin{proof}[Proof of Theorem~\ref{thm:representation}]
	We adopt the notation of \S~\ref{sec:quaddiffeq} and write
	\[ \Theta (z) \coloneqq \frac{z}{2} 
	\left(\frac{A'(z)}{A(z)}+\frac{A'(-z)}{A(-z)}\right)\]
	and 
	\[
	\psi(z) \coloneqq  A(z)A(-z) = \prod_{\mu\in \mathcal{Z}(g)}
	\left(1-\frac{z^2}{\mu^2}\right).
	\]
Then \eqref{eq:quadratic2} implies that
	\[
	\psi(z)\Theta(z) = -\frac{a}{2z}\, ,
	\]
and it follows that
	\begin{equation}\label{eq:semirep}
		f(z) \Theta(z) =
		-\frac{a f(0)}{z} + \sum_{\mu\in \mathcal{Z}(g)} \frac{\mu 
			z}{(\mu^2-z^2)} f(\mu)
			\end{equation}
	for every $f$ in $PW^1$. Indeed, we set
	\[
	F(z):= f(z)\Theta(z) +\frac{a f(0)}{z} - \sum_{\mu\in \mathcal{Z}(g)} 
	\frac{\mu
		z}{(\mu^2-z^2)} f(\mu)
	\]
	and see exactly as in the proof of Theorem~\ref{thm:charphi} that $\psi F$ is in $PW^1$ and then that $F\equiv 0$.
	
	We finally observe that all the arguments in the proof of Lemma~\ref{lem:frep} can be
	reversed, and from \eqref{eq:semirep} we obtain
	\[
	\sum_{\mu\in \mathcal{Z}(g)} \sin(\mu \xi) = \frac{a \xi}{2}
	\]
	in the distributional sense for all $-\pi< \xi < \pi$. Using again that $\mathcal{S}\cap PW^1$ is dense in $PW^1$, we 
	arrive at \eqref{eq:summ} by  Plancherel's identity.
\end{proof}

\section{Numerics}\label{sec:numerics}
Although the differential equation~\eqref{eq:diffbasic} completely determines $\Phi$ (and hence $\phi$), one needs to know the values $\Cone$ and $L_{\tau}(1)$ to start computing with it. Moreover, the recursion~\eqref{eq:recursion} is not numerically stable and to compute coefficients $[z^n]\Phi(z)$ for very large values of $n$, one needs to know the initial parameters $\Cone$ and $L_{\tau}(1)$ to high precision. From this datum, together with~\eqref{eq:functional},~\eqref{eq:taunformula},~\eqref{eq:phifunctional}, and~\eqref{eq:phifourier}, one can then compute any reasonable quantities related to $\phi$ and $\Phi$.

We now outline how one may compute $\Cone$ and $L_{\tau}(1)$ to any desired precision, exemplified by the 100 digit precision exhibited in \eqref{eq:precise100}.
We begin by giving a characterization of $\Cone$ and $L_{\tau}(1)$ in terms of the family of differential equations $\mathcal{L}_{a,b}$.
\begin{lemma}\label{lem:computeC}
	Let $a,b>0$, assume that $f$ is a $\lambda$-eigenfunction $f$ of $\mathcal{L}_{a,b}$ whose zeros are real and simple, and let $\kappa_{\pm}$ associated with $f$ be defined as in~\eqref{eq:feqgeneral}. Then the condition
	\begin{equation} \label{eq:admissibilitykappa}
		\kappa_{+}e^{i\pi/4}+\kappa_{-}e^{-i\pi/4} = 0
	\end{equation}
	is equivalent to $ab=\frac{\pi}{4\Cone}$, in which case $\lambda=-\frac{L_{\tau}(1)}{2\Cone}$.
\end{lemma}
\begin{proof}
	The conditions are invariant under rescaling $(a,b)\mapsto (a/k,bk)$, so we may assume that $b=\pi/2$. Then as in the proof of~\eqref{eq:functional}, by~\eqref{eq:feqgeneral} the condition~\eqref{eq:admissibilitykappa} is equivalent to the set of zeros of $f(z)$ being close to $2n-1/2+o(1)$, $n\to \pm \infty$. Then the summation formula from Theorem~\ref{thm:representation} involves an admissible sequence, and so we obtain \eqref{eq:L1ort} with $c=0$. As already noted in our discussion of \eqref{eq:L1ort}, this implies that  $f=\Phi$.
\end{proof}
It will be convenient to fix $b=1$ and use the same notation as in~\S~\ref{sec:spectrum}.
By Corollary~\ref{cor:fourier}, we have $\kappa_{\pm}=\pm 2\pi i\widehat{f}(\mp 1)$, so~\eqref{eq:admissibilitykappa} is equivalent to
	\begin{equation} \label{eq:condition2} \widehat{f}(1) = i\widehat{f}(-1)\end{equation}
and since Legendre polynomials satisfy $P_n(\pm1) = (\pm1)^n$, and $\widehat{f} = \sum_{n\ge0}i^n\xi_n P_n$ on $(-1,1)$, we may recast \eqref{eq:condition2}  as
	\begin{equation*} \label{eq:Cequation0}
	\sum_{n\ge0}i^n\xi_n = \sum_{n\ge0}i^{1-n}\xi_n,
	\end{equation*}
or equivalently, as
	\begin{equation} \label{eq:Cequation}
	\sum_{n\ge0}(-1)^{\lfloor (n-1)/2\rfloor}\xi_n = 0.
	\end{equation}
Here we normalize $f$ by setting $\xi_0=1$. Therefore, to compute $\Cone$ (and $L_{\tau}(1)/(2\Cone)$), it suffices to find $a$ and $\lambda$ for which an eigenvector $\xi$ of the tridiagonal matrix $T$ with entries~\eqref{eq:entries} in addition satisfies~\eqref{eq:Cequation}. The computation is now as follows: for each truncation $T_N$ of $T$ we find the value of $0<a<3/2$ for which the smallest eigenvector satisfies $\sum_{n=0}^{N}(-1)^{\lfloor (n-1)/2\rfloor}\xi_n=0$, and then take the limit as $N\to\infty$. Lemma~\ref{lem:specimproved} then guarantees exponentially fast convergence of $a$ to $\frac{\pi}{4\Cone}$, and $\lambda$ to $-\frac{L_{\tau}(1)}{2\Cone}$. Indeed, we get existence of an intermediate value by verifying that $\sum_{n=0}^{\infty}(-1)^{\lfloor (n-1)/2\rfloor}\xi_n$ changes sign when $a$ goes from $1.44$ to $1.46$. Moreover, it is clear that Lemma~\ref{lem:computeC} applies for $a$ in this range, since $f$ and hence the zeros of $f$ depend continuously on $a$, and the zeros of $\Phi$ are well separated. 

\section{Final remarks} \label{sec:final}
\subsection{Two formulas for \texorpdfstring{$\Cone$}{C}}
\label{sec:formulasC}
We now establish the two identities in \eqref{eq:expressC}. We begin by noting that the second formula 
\[ \frac{1}{\Cone}=2+4\sum_{n=1}^\infty (-1)^n \Big(n+\frac12-\tau_n\Big) \] 
follows at once if we set $\xi=0$ in \eqref{eq:sumC}. To prove the product formula in \eqref{eq:expressC}, we start from 
Theorem~\ref{thm:charphi} which shows that
\[ \phi'(\tau_n)=\frac{(-1)^n}{2\Cone \tau_n^2}, \]
and so
\[  \prod_{m\neq n} \left|1-\frac{\tau_n^2}{\tau_m^2}\right| = \frac{1}{4 \Cone \tau_n}.\]
On the other hand, setting
\[ D\coloneqq \prod_{n=1}^{\infty}\frac{(n+\frac12)^2}{\tau_n^2}, \]
we see that
\[  \prod_{m\neq n} \left|1-\frac{\tau_n^2}{\tau_m^2}\right|=
\frac{D\tau_n^2}{(n+\frac12)^2} \left|1-\frac{\tau_n^2}{(n+\frac12)^2}\right|^{-1}
(4\tau_n^2-1)^{-1} |\cos (\pi \tau_n)| \prod_{m\neq n}
\left|\frac{\tau_m^2-\tau_n^2}{(m+\frac12)^2-\tau_n^2}\right|,\]
so that
\[ \frac{1}{4D \Cone}= \frac{\tau_n^3}{(\tau_n+n+\frac12)(4\tau_n^2-1)}  \frac{|\cos (\pi
\tau_n)|}{|n+\frac12-\tau_n|} \prod_{m\neq n}
\left|\frac{\tau_m^2-\tau_n^2}{(m+\frac12)^2-\tau_n^2}\right|   .\]
Letting $n$ tend to $\infty$, we arrive at the formula
\[ \frac{2}{\pi \Cone}=\prod_{n=1}^\infty \frac{(n+\frac12)^2}{\tau_n^2}. \]
Using finally the Wallis product, we obtain the desired identity 
	\[\Cone = \frac{1}{2}\prod_{n=1}^\infty \frac{\tau_n^2}{n(n+1)}. \]

\subsection{Associated Dirichlet series}
We record several observations about the Dirichlet series
	\[L_{+}(s) \coloneqq \sum_{n=1}^{\infty}\frac{1}{\tau_n^s}\qquad\mbox{and}\qquad L_{-}(s) \coloneqq \sum_{n=1}^{\infty}\frac{(-1)^n}{\tau_n^s},\]
some of whose properties are reminiscent of the more familiar classical $L$-functions.
Note that $L_{-}(s)$ was previously denoted by $L_{\tau}(s)$, when we needed to distinguish $\tau$ from a generic admissible sequence $t$. 
Defined initially for $\re(s) > 1$, both series extend meromorphically to $\C$. Explicitly, we use expansion~\eqref{eq:taunformula} of $\tau_n$ to write
	\[
	L_{\pm}(s) =
	\sum_{n\ge1}\frac{(\pm1)^n}{(n+1/2)^s}+\sum_{n\ge1}\frac{(\pm1)^nsa_1}{(n+1/2)^{s+2}}
	+\sum_{n\ge1}\frac{(\pm1)^n(2a_3s+a_1^2s(s+1))}{2(n+1/2)^{s+4}}+\dots\,,
	\]
and then use $\sum_{n\ge 1}\frac{1}{(n+1/2)^{s}} =
2^s(\zeta(s)(1-2^{-s})-1)$, and $\sum_{n\ge1}\frac{(-1)^n}{(n+1/2)^{s}} = 2^s(L(s,\chi_{4})-1)$, where~$\chi_4$ is the nontrivial primitive character modulo $4$. From this calculation it follows that $L_{+}(s)$ is meromorphic, with simple poles at $1-2k, k\ge0$, whereas $L_{-}(s)$ is entire. 

The special values of $L_{\pm}$ at positive integers have already figured in~\S\ref{sec:quaddiffeq}, namely,
	\begin{align*} 
	\Theta_{\tau}(z) &=-\frac{1}{4\Cone z}+\sum_{m=1}^\infty L_{-}(2m-1) z^{2m-1}\,,\\
	\psi_{\tau}(z) &= \exp\Big(-\sum_{k=1}^{\infty} L_{+}(2k)\frac{z^{2k}}{k}\Big)\,.
	\end{align*}
The identity
	\[\psi_\tau(z) \Theta_\tau(z)=-\frac{1}{4\Cone z}\]
allows one to write $L_{+}(2k)$ via $L_{-}(2m-1)$ and vice versa (showing, in particular, that all these numbers lie in $\Q[\Cone,\pi, L_{\tau}(1)]$). For example, the identity $L_{+}(2) = -4\Cone L_{-}(1)$, written out as
	\[\sum_{n\ge1}\frac{1}{\tau_n^2} = -4\Cone \sum_{n\ge1}\frac{(-1)^n}{\tau_n}\]
can be viewed as an analogue of the formula 
	\[\sum_{n\ge1}\frac{1}{(n-1/2)^2}=-\pi \sum_{n\ge1} \frac{(-1)^n}{(n-1/2)}\] 
relating Euler's solution of the Basel problem and Leibniz's formula for $\pi$.

Going further to values of $L_{\pm}$ at negative integers, we obtain another identity connecting the special values of $L_{+}$ and $L_{-}$:
	\[\mathrm{res}_{s=1-2k}   L_{+}(s) = \frac{2i}{\pi}\frac{L_{-}(2k-1)}{(2\pi i \Cone)^{2k-1}}\,.\]
This is not hard to prove using~\eqref{eq:taunformula} together with Theorem~\ref{thm:zeroexp}.

Next, for the values of $L_{-}$ at negative integers, we claim that~\eqref{eq:dercond} implies that
	\begin{equation} \label{eq:Lodd} 
		L_{-}(-1)=-\frac{1}{4\Cone},\qquad L_{-}(-1-2m)=0,\qquad m=1,2,\dots\, ,
	\end{equation}
which is consistent with our declaration in  \S~\ref{sec:quaddiffeq}  that $L_{\tau}(-1)\coloneqq -\frac{1}{4\Cone}$. 
We may prove \eqref{eq:Lodd} in a familiar way using the Mellin transformation.  Indeed, setting
	\[S(t) \coloneqq 2\sum_{n\ge1}(-1)^n\tau_ne^{-\tau_n^2 t}\, \]
we find that 
	\[\int_{0}^{\infty}xS(x^2)x^{s-1}dx = \sum_{n\ge1}(-1)^{n}\int_{0}^{\infty}x\tau_ne^{-x^2\tau_n^2}x^{2(s/2-1)}dx^2
	= \Gamma(\tfrac{s+1}{2})L_{-}(s)\,.\]
Since $S(t)$ is a pairing between a tempered distribution $\Psi\coloneqq \sum_{n}(-1)^n\tau_n(\delta_{\tau_n}+\delta_{-\tau_n})$ and the Gaussian $x\mapsto e^{-t x^2}$, and since $\widehat{\Psi}+\frac{1}{4\Cone}$ vanishes in $(-\pi,\pi)$ by Lemma~\ref{lem:summform}, we get that $S(1/t) = -\frac{1}{2\Cone} + O(e^{-ct})$ for any $c<\pi^2/4$ when $t\to \infty$. It follows that
\[  \Gamma(\tfrac{s+1}{2})L_{-}(s) = - \frac{1}{4\Cone (s+1)} + \text{an entire function}, \]
and this implies~\eqref{eq:Lodd}.

Finally, for the values of $L_{+}$ at negative even integers,  we conjecture the following symmetry, verified numerically to very high precision.
\begin{conjecture}
	The values $L_{+}(2k)$ satisfy the following symmetry:
	\[L_{+}(-2k) = \frac{L_{+}(2k)}{(2\pi i\Cone)^{2k}}\,,\qquad k\in \Z\,.\]
\end{conjecture}

\subsection{An integrality phenomenon}
Looking at the Taylor series of $\phi$ and $\Phi$, we may note that whereas $[z^n]\Phi(z)$ lies in $\Q[\pi, L_{\tau}(1), \Cone]$ and has coefficients with large denominators as $n\to\infty$, the coefficients $[z^n]\phi(z)$ appear to lie in $\Z[\pi, L_{\tau}(1), \Cone]$. This observation can be made more precise if we consider the family of differential operators from \S~\ref{sec:family}, leading to the following conjecture.
\begin{conjecture} \label{conj:integrality}
Let $u_n\in\Q[b,\lambda]$ be a sequence defined by the recursion
	\[u_{n+1}=\frac{4n+2}{n+1}u_{n}(n(n+1)-\lambda)+b^2\frac{4n}{n+1}u_{n-1},\qquad n\ge0,\]
where $u_{-1}=0$ and $u_0=1$. Then $u_{n}\in\Z[b,\lambda]$ for $n\ge0$.
\end{conjecture} 
The above recursion is satisfied by the coefficients of $f(z)f(-z)=\sum_{n\ge0}u_nz^{2n}$, where~$f$ is any eigenfunction of $\mathcal{L}_{1,b}$. We may compare the above integrality property with Ap\'ery's proof of irrationality of~$\zeta(3)$, where a similar phenomenon plays an important role (see~\cite{zagier2009integral} for an analysis of Ap\'ery-like recurrences with integrality properties). Note that in Ap\'ery-like recursions the associated differential equation is a period equation for a family of algebraic varieties, so it has regular singularities, whereas in Conjecture~\ref{conj:integrality} the differential equation has irregular singularities. It would be interesting to know what explains the integrality in this case.

\bibliographystyle{amsplain}
\bibliography{pwpeval}

\end{document}